\documentclass[11pt]{article}

\usepackage{amsfonts,amssymb}
\usepackage{diagrams}
\usepackage{pstricks,multido}

\newarrow{Dots}{}...{}

\def\Rep{{\rm Rep}}
\def\ac{{\mathcal{AC}}}

\def\Schub{{\mathcal S}}
\def\lleq{\prec}
\def\Orb{{\mathcal O}}
\def\longratto{\longrightarrow} 

\makeatletter
\def\revddots{\mathinner{\mkern1mu\raise\p@\vbox{\kern7\p@\hbox{.}}\mkern2mu\raise4\p@\hbox{.}\mkern2mu\raise7\p@\hbox{.}\mkern1mu}}
\makeatother

\newcommand\eqref[1]{(\ref{#1})}

\newenvironment{proof}{{\noindent\bf Proof.}}{\hfill $\square$}

\newcounter{paragrafsubsub}[subsubsection]
\renewcommand{\theparagrafsubsub}{%
\thesubsubsection.\roman{paragrafsubsub}}
\newcommand{\paragrafsubsub}{%
\refstepcounter{paragrafsubsub}
{\bf \theparagrafsubsub}\hspace{0.2em}--- }

\newcounter{paragrafsub}[subsection]
\renewcommand{\theparagrafsub}{\thesubsection.\arabic{paragrafsub}}
\newcommand{\paragrafsub}{%
\refstepcounter{paragrafsub}
{\bf \theparagrafsub}\hspace{0.2em}--- }

\newcounter{paragraf}[section]
\renewcommand{\theparagraf}{\thesection.\arabic{paragraf}}
\newcommand{\paragraf}{%
\refstepcounter{paragraf}
{\bf \theparagraf}\hspace{0.2em}--- }

\newcommand\paragraphe{%
\par \indent
\ifcase\value{subsection} %
\paragraf
\else
\ifcase\value{subsubsection}\paragrafsub %
\else\paragrafsubsub
\fi\fi
}

\def\longto{\longrightarrow}
\def\div{{\rm div}}

\def\Part{{\mathcal P}}\def\lr{{\mathcal LR}}
\def\Sch{{\mathcal S}}
\def\PP{{\mathbb P}}
\def\RR{{\mathbb R}}\def\QQ{{\mathbb Q}}\def\ZZ{{\mathbb Z}}\def\NN{{\mathbb N}}
\def\CC{{\mathbb C}}

\def\Fl{{\mathcal Fl}}\def\Face{{\mathcal F}}
\def\Gr{{\mathbb G}}
\def\Sl{{\rm Sl}}
\def\IG{{\Gr_\omega}}
\def\SO{{\rm SO}}

\def\Pic{\rm Pic}

\def\base{{\mathcal B}}
\def\diag{{\rm diag}}

\def\lh{{\mathfrak h}}\def\lu{{\mathfrak u}}
\def\lu{{\mathfrak u}}
\def\lt{{\mathfrak t}}

\def\kbprod{{\odot_0}}

\def\SL{{\rm SL}}\def\Sp{{\rm Sp}}
\def\SO{{\rm SO}}

\def\Hom{{\rm Hom}}
\def\tr{{\rm tr}}

\def\Li{{\mathcal{L}}}
\def\Mi{{\mathcal{M}}}

\def\quot{/\hspace{-.5ex}/}

\def\GL{{\rm GL}}

\newtheorem{lemma}{Lemma}
\newtheorem{prop}{Proposition}
\newtheorem{theo}{Theorem}
\newtheorem{coro}{Corollary}

\newtheorem{defin}{Definition}
\newtheorem{remark}{Remark}

\begin{document}
\title{A cohomology free description of eigencones\\ 
in type A, B and C}
\author{N. Ressayre}

\maketitle
\begin{abstract}
Let $K$ be a compact connected  Lie group.
The triples $(\Orb_1,\,\Orb_2,\,\Orb_3)$ of adjoint $K$-orbits such that 
 $\Orb_1+\Orb_2+\Orb_3$ contains $0$ are parametrized by a closed
 convex polyhedral cone, called the eigencone of $K$.
For $K$ simple of type $A$, $B$ or $C$ we give an inductive cohomology free parametrization 
of the minimal set of linear inequalities which characterizes the
eigencone of $K$.
\end{abstract}

\section{Introduction}

\paragraphe
We first explain the Horn conjecture which answers the following elementary
question:
\begin{center}
What can be said about the eigenvalues of a sum of two Hermitian matrices, 
in terms of the eigenvalues of the summands?
\end{center}

If $A$ is a Hermitian $n$ by $n$ matrix, we will denote by 
$\lambda(A)=(\lambda_1\geq\cdots\geq\lambda_n)\in \RR^n$ its spectrum. 
Consider the following set:
$$
{\rm Horn}_\RR(n)=\{(\lambda(A),\lambda(B),\lambda(C))\in\RR^{3n}\,:\,
\begin{array}{l}

A, B, C {\rm\ are\ 3\ Hermitian\ matrices}\\
{\rm s.t.\ }A+B+C=0\}.  
\end{array}
$$
It turns out that ${\rm Horn}_\RR(n)$ is a closed convex polyhedral cone
in $\RR^{3n}$. 
We now want to explain the Horn conjecture which describes inductively a  list of linear inequalities which 
characterizes this cone.
Let $\Part(r,n)$ denote the set of parts of $\{1,\cdots,n\}$ with $r$ elements. 
Let $I=\{i_1<\cdots<i_r\}\in\Part(r,n)$. 
We set:
$\lambda^I =(i_r- r, i_{r-1}- (r- 1), \cdots , i_2- 2, i_1- 1).$
We will denote by $1^r$ the vector $(1,\cdots,1)$ in $\RR^r$.

\begin{theo}\label{th:Hornconj}
  Let $(\lambda,\mu,\nu)$ be a triple of non-increasing sequences of $n$ real numbers.
Then, $(\lambda,\mu,\nu)\in{\rm Horn}_\RR(n)$ if and only if 
\begin{eqnarray}
  \label{eq:trace}
  \sum_i\lambda_i+\sum_j\mu_j+\sum_k\nu_k=0
\end{eqnarray}
 and
for any $r=1,\cdots,n-1$, for any $(I,J,K)\in \Part(r,n)^3$ such that 
\begin{eqnarray}
  \label{eq:induction}
(\lambda^I,\lambda^J,\lambda^K-2(n-r)1^r)\in{\rm Horn}_\RR(r),
\end{eqnarray}

we have:
\begin{eqnarray}
  \label{eq:ineqIJK}
\sum_{i\in I}\lambda_i+\sum_{j\in J}\mu_j+\sum_{k\in K}\nu_k\leq 0.
\end{eqnarray}
\end{theo}

Note that if starting with a point in ${\rm Horn}_\RR(r)$, one adds $1^r$ to one 
factor add $-1^r$ to another one, one stays in ${\rm Horn}_\RR(r)$. This remark
implies that condition~\eqref{eq:induction} is symmetric in $I$, $J$ and $K$.

In 1962, Horn \cite{Horn:conj} conjectured Theorem~\ref{th:Hornconj}.
This conjecture was proved by combining works by Klyachko \cite{Kly:stable} and
Knutson-Tao \cite{KT:saturation} (see also \cite{Fulton:survey} for a survey).  
Despite the proof,  the statement of Theorem~\ref{th:Hornconj} is as elementary
as the Horn problem is. 
Note that $I$, $J$ and $K$ are  sets of indexes in inequality~\eqref{eq:ineqIJK} 
whereas $\lambda^I$, $\lambda^J$ and $\lambda^K$ are eigenvalues of Hermitian 
matrices in condition~\eqref{eq:induction}.
This very curious remark certainly contributed  to the success of the Horn conjecture.\\

As pointed out by C. Woodward, Theorem~\ref{th:Hornconj} has a weakness. 
Indeed, it gives redundant inequalities. To describe a minimal set of inequalities,
we need to introduce some notation.
Let $\Gr(r,n)$ be the Grassmann variety of $r$-dimensional subspaces  of  $\CC^n$.
Consider its cohomology ring ${\rm H}^*(\Gr(r,n),\ZZ)$. 
To any $I\in\Part(r,n)$ is associated a Schubert class 
$\sigma_I\in{\rm H}^*(\Gr(r,n),\ZZ)$. 
There are two usual ways, obtained one from each other composing by
Poincar\'e duality, to assign a Schubert class $\sigma_I$ to $I$. Our
choice is detailed in Paragraph \ref{par:SchubGr}.
Let $[{\rm pt}]\in{\rm H}^{2r(n-r)}(\Gr(r,n),\ZZ)$ denote the Poincar\'e dual class
of the point. Belkale proved in \cite{Belk:P1} the following:

\begin{theo}\label{th:Belk}
  Let $(\lambda,\mu,\nu)$ be a triple of non-increasing sequences of $n$ real numbers.
Then, $(\lambda,\mu,\nu)\in{\rm Horn}_\RR(n)$ if and only if 
\begin{eqnarray}
\label{eq:traceBelk}
  \sum_i\lambda_i+\sum_j\mu_j+\sum_k\nu_k=0
\end{eqnarray}
 and
for any $r=1,\cdots,n-1$, for any $(I,J,K)\in \Part(r,n)^3$ such that 
\begin{eqnarray}
  \label{eq:cohomcond}
\sigma_I.\sigma_J.\sigma_K=[{\rm pt}]\in {\rm H}^*(\Gr(r,n),\ZZ),
\end{eqnarray}

we have:
\begin{eqnarray}
  \label{eq:ineqIJKBelk}
\sum_{i\in I}\lambda_i+\sum_{j\in J}\mu_j+\sum_{k\in K}\nu_k\leq 0.
\end{eqnarray}
\end{theo}
 
The statement of Theorem~\ref{th:Belk} is not elementary, but as proved by Knutson-Tao-Woodward in 
\cite{KTW} it is optimal:

\begin{theo}\label{th:KTW}
  In Theorem~\ref{th:Belk}, no inequality can be omitted.
\end{theo}

\bigskip\paragraphe
Klyachko obtained in \cite{Kly:stable} a variation of Theorem
\ref{th:Belk}:

\begin{theo}\label{th:Kly}
  Theorem \ref{th:Belk} holds replacing condition \eqref{eq:cohomcond}
  by
\begin{eqnarray}
  \label{eq:cohomcond2}
\sigma_I.\sigma_J.\sigma_K=c[{\rm pt}], {\rm\ for\ a\ nonzero\
  integer\ c.}
\end{eqnarray}
\end{theo}

If $\lambda$, $\mu$ and $\nu$ are three partitions then we denote by
$c_{\lambda\,\mu}^\nu\in\ZZ^+$ the associated Littlewood-Richardson
coefficient.
These integers control the cup product in
the cohomology of the Grassmanians.  For example, condition
\eqref{eq:cohomcond2} is equivalent to the nonzeroness of some
Littlewood-Richardson coefficient. 

If $\nu=(\nu_1\geq\cdots\geq \nu_n)$, we set $\nu^\vee=(-\nu_n\geq\cdots\geq -\nu_1)$.
A consequence of Knutson-Tao's saturation theorem is that

\begin{theo}
  \label{th:KTsat}
  The coefficient $c_{\lambda\,\mu}^\nu$ is nonzero if and only if  
$(\lambda,\mu,\nu^\vee)\in {\rm Horn}_\RR(n)$.
\end{theo}

Compare  Theorems \ref{th:Belk} and \ref{th:Kly}. The advantage of
Theorem \ref{th:Belk} is obvious: it gives exactly the minimal set
of linear inequalities needed to characterized ${\rm Horn}_\RR(n)$.
The advantage of Theorems \ref{th:Kly} is that by Theorem
\ref{th:KTsat},  condition \eqref{eq:cohomcond2} can be
reinterpreted in terms of Horn cones.
 This is the key point to explain the inductive nature of Theorem
 \ref{th:Hornconj}. Finally,  Theorem
 \ref{th:Hornconj} gives an inductive algorithm to decide if a given
 point belongs to ${\rm Horn}_\RR(n)$. Equivalently, it gives an inductive
 algorithm to decide if a given Littlewood-Richardson is zero or positive.

\bigskip\paragraphe
In this work, we give an 
{\bf inductive algorithm to decide if a given Littlewood-Richardson coefficient equals 
to one or not}.
More precisely, our algorithm proceeds inductively and decides if a
given Littlewood-Richardson coefficient is equal to zero or if it is
equal to one or  it  is greater than one.
In particular, our algorithm decides if condition~\eqref{eq:cohomcond} is fulfilled.
The combination of this algorithm with Theorems~\ref{th:Belk} and \ref{th:KTW}
gives an  inductive description of the minimal set of inequalities 
of ${\rm Horn}_\RR(n)$.
Note that our algorithm uses  Derksen-Weyman's one (see~\cite{DW:algo}) as a procedure.

Our algorithm will be detailed in  Section~\ref{sec:algo}.
 We now explain the main point of the algorithm. 
First, Horn's conjecture allows to decide if a given Littlewood-Richardson 
coefficient is zero or not. 
So, the remaining question is to decide if a  given Littlewood-Richardson 
coefficient is less or equal to one or not.

First, we have a geometrical construction to obtain 
 such coefficients.
Indeed, consider a product $Y=G/P\times G/Q\times G/R$ of three compact 
$G=\GL_n(\CC)$-homogeneous spaces which contains a dense $G$-orbit.
Let $\Li$ be a $G$-linearized line bundle on $Y$. 
Then, the dimension of the space $H^0(Y,\Li)^G$ of $G$-invariant sections
 of $\Li$ is a Littlewood-Richardson coefficient $c$. 
The fact that $Y$ contains a dense $G$-orbit implies that
\begin{eqnarray}
  \label{eq:homdonne1}
  c\leq 1.
\end{eqnarray}

On the other hand, a Derksen-Weyman's theorem (see \cite{DW:comb,Roth:red,ReductionRule}) shows that some Littlewood-Richardson  coefficients $c$ are  equal to a product of two other ones $c_1$ and $c_2$ 
(associated to smaller partitions); $c=c_1.c_2$.
The obvious remark is that 
\begin{eqnarray}
  \label{eq:DWdonne1}
  c=1\ \iff\  c_1=c_2=1.
\end{eqnarray}
Roughly speaking, our algorithm is based on the fact that {\it 
any Littlewood-Richardson coefficient $c$ equal to one,  can be obtained 
applying finitely many times assertions~\eqref{eq:homdonne1} and 
\eqref{eq:DWdonne1}}.
In particular, the LR-coefficients associated to regular partitions equal to 
 one can be obtained using the fact that $SL_2(\CC)$ has a dense orbit 
in $(\PP^1)^3$ and Derksen-Weyman's theorem.

\bigskip\paragraphe
We now want to explain a generalization of the Horn problem.
Let $G$ (e.g. $G=\GL_n(\CC)$) be a reductive complex group and  $U$
(e.g. $U=U_n(\CC)$) be a maximal compact subgroup.
Let $\lu$ denote its Lie algebra.
We are interested in the following problem:
what are the triples $(\Orb_1,\Orb_2,\Orb_3)$ of adjoint orbits such that 
$\Orb_1+\Orb_2+\Orb_3$ contains $0$.

Let $T$ be a maximal torus of $G$ such that  $T\cap U$ is a Cartan subgroup of $U$.
Let $\lt$ denote its  Lie algebras and $\lt^+$ be a fixed Weyl chamber of $\lt$.
It turns out that the triples of orbits as above are  parametrized by a closed convex 
polyhedral cone contained in $(\lt^+)^3$ (see Section~\ref{sec:eigencone} for details).
We will denote by $\Gamma(U)$ this cone.
Using the Cartan-Killing form one can identify $\Gamma(U(n))$ with
${\rm Horn}_\RR(n)$.\\

We now introduce notation to describe a minimal set of inequalities for $\Gamma(U)$.

Let $\alpha$ be a simple root of $G$ and let  $\omega_\alpha$ be the
corresponding fundamental weight.
We consider the  standard maximal parabolic subgroup 
$P_\alpha$ associated to $\alpha$.
Let $W$ denote the Weyl group of $G$. The Weyl group $W_\alpha$ of $P_\alpha$ is also
the stabilizer of $\omega_{\alpha}$.

Consider now the cohomology group ${\rm H}^*(G/P_\alpha,\ZZ)$: it is freely generated
by the Schubert classes $\sigma_w$ parametrized by the cosets $w\in W/W_\alpha$.
In \cite{BK}, Belkale-Kumar defined a new product denoted $\kbprod$ on 
${\rm H}^*(G/P_\alpha,\ZZ)$. We can now state the main result of \cite{BK} which 
generalizes Theorem~\ref{th:Belk}:

\begin{theo}\label{thi:BK}
We assume that $U$ is semisimple.
  Let $(\xi,\,\zeta,\,\eta)\in(\lt^+)^3$. Then, $(\xi,\,\zeta,\,\eta)$ belongs to 
$\Gamma(U)$ if and only if for any simple root $\alpha$ and any triple of
Schubert classes $\sigma_u,\,\sigma_v$ and $\sigma_w$ in ${\rm H}^*(G/P_\alpha,\ZZ)$
such that
\begin{eqnarray}
  \label{eq:BKcond}
  \sigma_u\,\kbprod\,\sigma_v\,\kbprod\,\sigma_w=[{\rm pt}],
\end{eqnarray}
we have:
\begin{eqnarray}
  \label{eq:ineqGamma}
 \omega_{\alpha}(u^{-1}\xi)+\omega_{\alpha}(v^{-1}\zeta)+\omega_{\alpha}(w^{-1}\eta)\leq 0.
\end{eqnarray}
\end{theo}

In \cite{GITEigen}, the following generalization of Theorem~\ref{th:KTW} is
obtained:

\begin{theo}\label{th:GITEigen}
 In Theorem~\ref{thi:BK}, no inequality can be omitted.  
\end{theo}

\bigskip\paragraphe {\bf For $U$ simple of type B or C}, in Theorems~\ref{th:ppalIG} and 
\ref{th:ppalSO} below, we prove that each condition~\eqref{eq:BKcond} is equivalent to
the fact that two  Littlewood-Richardson coefficients 
(for ordinary grassmanian !) are equal to one.
The combination of Algorithm~\ref{sec:algo} and these results gives a
{\bf cohomology free   description of the minimal set of inequalities for $\Gamma(U)$}.
Note that in \cite{BK:typeBC}, Belkale-Kumar gave a redundant cohomology free description 
of $\Gamma(U)$.

\bigskip\paragraphe
The paper is organized as follows.
In Section~\ref{sec:Horn}, we introduce basic material about the Littlewood-Richardson coefficients and
the Horn cone.
In Section~\ref{sec:quiver}, we recall some useful results about quiver representations. 
 In Section~\ref{sec:algogen}, we state and prove our inductive algorithm to decide if
a given  Littlewood-Richardson coefficient equals to one or not.
In Section~\ref{sec:schub}, we introduce a parametrization of the Schubert classes of any complete 
rational homogeneous space and give some examples. In Section~\ref{sec:levimov}, we recall from \cite{BK}
the notion of Levi-movability. 
In Section~\ref{section:Eigencone}, we recall some results about eigencones.
In Sections~\ref{sec:cohomSp} and \ref{sec:cohomSO}, we prove our results about the cohomology
of isotropic and odd orthogonal Grassmannians.

\bigskip
\noindent{\bf Acknowledgments.}
I thank N. Perrin and M. Brion for useful discussions.
\section{The Horn cone}
\label{sec:Horn}

\subsection{The Littlewood-Richardson coefficients}

\paragraphe{\bf Schubert Calculus.}
\label{par:SchubGr}
Let $\Gr(r,n)$ be the Grassmann variety of $r$-dimensional subspaces $L$ of a 
fixed $n$-dimensional vector space $V$.
Let $F_\bullet$: $\{0\}=F_0\subset F_1\subset 
F_2\subset\cdots\subset F_{n}=V$ be a complete flag of $V$.

If $a\leq b$, we will denote by $[a;b]$ the set of integers between $a$ and $b$.
Let $\Part(r,n)$ denote the set of subsets of $[1;n]$ with $r$ elements.
For any $I=\{i_1<\cdots<i_r\}\in\Part(r,n)$,
the Schubert variety $\Omega_I(F_\bullet)$ in $\Gr(r,n)$  is defined by
$$
\Omega_I(F_\bullet)=\{L\in\Gr(r,n)\,:\,
\dim(L\cap F_{i_j})\geq j {\rm\ for\ }1\leq j\leq r\}.
$$
The Poincar\'e dual of the homology class of $\Omega_I(F_\bullet)$ does not depend on $F_\bullet$;
it is denoted by $\sigma_I$.
The $\sigma_I$'s form a $\ZZ$-basis for the cohomology ring of $\Gr(r,n)$. 
The class associated to $[1;r]$ is the class of the point; it will be denoted by $[{\rm pt}]$.
It follows that for any subsets $I,\,J\in\Part(r,n)$, there is a unique expression
$$
\sigma_I.\sigma_J=\sum_{K\in \Part(r,n)}c_{IJ}^K\sigma_K,
$$
for integers $c_{IJ}^K$. We define $K^\vee$ by: $i\in K^\vee$ if and only if $n+1-i\in K$.
Then, $\sigma_K$ and $\sigma_{K^\vee}$ are Poincar\'e dual.
So, if the sum of the codimensions of $\Omega_I(F_\bullet)$, $\Omega_J(F_\bullet)$ and 
$\Omega_K(F_\bullet)$ equals the dimension of $\Gr(r,n)$,  we have
$$
\sigma_I.\sigma_J.\sigma_K=c_{IJ}^{K^\vee}[{\rm pt}].
$$
We set 
$$
c_{IJK}:=c_{IJ}^{K^\vee}.
$$
Note that $c_{IJK}=c_{JIK}=c_{IKJ}=\dots$

\bigskip\paragraphe
Recall that the irreducible representations of $G=\GL_n(\CC)$ are indexed by sequences 
$\lambda=(\lambda_1\geq\cdots\geq\lambda_n)\in\ZZ^n$. 
Let us denote by $\Lambda_n^+$ the set of such sequences.
We set $|\lambda|=\lambda_1+\cdots+\lambda_n$.
Denote   the representation corresponding to $\lambda$ by $V_\lambda$.
For example, the representation $V_{1^n}$ is the determinant representation of $\GL_n(\CC)$.
Define the Littlewood-Richardson coefficients $c_{\lambda\,\mu}^\nu\in\NN$ by:
$$
V_\lambda\otimes V_\mu=\sum_{\nu\in\Lambda^+_n} c_{\lambda\,\mu}^\nu V_\nu.
$$
For $\nu=(\nu_1\geq\cdots\geq\nu_n)$, we set: 
$\nu^\vee=(-\nu_n\geq\cdots\geq -\nu_1)$. Then, 
$V_{\nu^\vee}$ is the dual of $V_\nu$. Finally, we set
$$c^n_{\lambda\mu\nu}=c_{\lambda\mu}^{\nu^\vee}.$$
Note that $c^n_{\lambda\mu\nu}$ is the dimension of the subspace
$(V_\lambda\otimes V_\mu\otimes V_\nu)^{G}$ of $G$-invariant vectors in 
$V_\lambda\otimes V_\mu\otimes V_\nu$.
Consider:
$$
{\rm Horn}(n):=\{(\lambda,\mu,\nu)\in(\Lambda_n^+)^3\,:\,c^n_{\lambda\mu\nu}\neq 0\}.
$$

\bigskip\paragraphe \label{par:lambdaI}
We will use the standard correspondence between elements $I=\{i_1<\cdots<i_r\}$ of $\Part(r,n)$ and 
partitions  $\lambda^I\in\Lambda_r^+$
such that $\lambda_1\leq n-r$ and $\lambda_r\geq 0$.
This correspondence is obtained by defining
$$
\lambda^I =(i_r- r, i_{r-1}- (r- 1), \cdots , i_2- 2, i_1- 1).
$$
Note that the dimension of $\Omega_I$ equals $|\lambda^I|$ and that
the degree of $\sigma_I$ is $2(r(n-r)-|\lambda^I|)$.
Authors use the base $\theta_I=\sigma_I^\vee$ for the cohomology group
of the Grassmanian. In this case, the degree of $\theta_I$ is
$2|\lambda^I|$.  
For $I,\,J$ and $K$ in $\Part(r,n)$, Lesieur showed in 1947 (see \cite{Lesieur}) that:
$$
\theta_I.\theta_J.\theta_{K^\vee}=c_{\lambda^I\lambda^J}^{\lambda^K}[{\rm
  pt}].
$$
Applying Poincar\'e duality, we deduce that 
$$
c_{IJK}=c^r_{\lambda^I\lambda^J\lambda^K-2(n-r)1^r}.
$$




\subsection{The Horn cone}

Let $I=\{i_1<\cdots<i_r\}\in\Part (r,n)$ and $\lambda\in\Lambda^+_n$.
We set 
$$
\lambda_I=(\lambda_{i_1}\geq\cdots\geq\lambda_{i_r}).
$$
In particular, $|\lambda_I|=\sum_{i\in I}\lambda_i$.
Let $I,J,K\in\Part(r,n)$. We define the ``linear form'' $\varphi_{IJK}$
 on $(\Lambda_n^+)^3$ by:
$$
\varphi_{IJK}(\lambda,\mu,\nu)=|\lambda_I|+|\mu_J|+|\nu_K|.
$$

Combining \cite{Belk:P1} and \cite{KT:saturation}, we obtain the following description of
${\rm Horn}(n)$:

\begin{theo}
\label{th:Horn}
Let $(\lambda,\mu,\nu)\in(\Lambda_n^+)^3$.
The point $(\lambda,\mu,\nu)$ belongs to ${\rm Horn}(n)$ if and only if 
$$
|\lambda|+|\mu|+|\nu|=0,
$$
and for any $r\in [1;n-1]$, for any $(I,J,K)\in\Part(r,n)$ such that
$\sigma_I\sigma_J\sigma_K=[{\rm pt}]$, we have:
$$
\varphi_{IJK}(\lambda,\mu,\nu)\leq 0.
$$
\end{theo}

\section{Quivers}
\label{sec:quiver}

In this section, we  explain how a Derksen-Weyman's algorithm 
(see \cite{DW:algo}) on quiver representations 
can be applied to decide if a product of three flag manifolds
contains a dense orbit or not.
Let us first introduce standard notation on quivers.

\subsection{Definitions}
\label{sec:defcarquois}

Let $Q$ be a quiver (that is, a finite oriented graph) with vertexes $Q_0$ and arrows $Q_1$.
We assume that $Q$ has no oriented cycle.
An arrow $a\in Q_1$ has initial vertex $ia$ and terminal one $ta$.
A representation $R$ of $Q$ is a family $(V(s))_{s\in Q_0}$ of finite dimensional vector spaces and 
a family of linear maps $u(a)\in {\rm Hom}(V(ia),V(ta))$ indexed by $a\in Q_1$.
The dimension vector of $R$ is the family $(\dim(V(s)))_{s\in Q_0}\in \NN^{Q_0}$.

Let us fix $\alpha\in \NN^{Q_0}$ and  a vector space $V(s)$ of dimension $\alpha(s)$ for each $s\in Q_0$. 
Set
$$
\Rep(Q,\alpha)=\bigoplus_{a\in Q_1}{\rm Hom}(V(ia),V(ta)).
$$
The group
$
\GL(\alpha)=\prod_{s\in Q_0}\GL(V(s))
$
acts naturally on $\Rep(Q,\alpha)$.

\subsection{Our question in terms of quivers}

Let $\{a_1<\cdots<a_p\}$ be a part of $[1;n-1]$.
We will consider the following flag variety:
$$
\Fl_n(a_1,\cdots,a_p):=\{(V_1\subset\cdots\subset V_p)\}\subset\Gr(a_1,n)\times
\cdots\times\Gr(a_p,n).
$$
Let also $\{b_1<\cdots<b_q\}$ and  $\{c_1<\cdots<c_r\}$ be two other parts
 of $[1;n]$. Recall that we want to decide if
$\Fl_n(a_1,\cdots,a_{p})\times\Fl_n(b_1,\cdots,b_{q})\times\Fl_n(c_1,\cdots,c_{r})$ contains a dense $\GL_n$-orbit.

Consider the following quiver $T_{pqr}$ with $p+q+r+1$ vertexes and  $p+q+r$ arrows:

\begin{diagram}[size=2em]
   &    &x_1&\rTo&x_2&\rTo&x_3&\rDots&x_{p}\\
   &    &   &    &   &    &   &      &       &\rdTo\\
y_1&\rTo&y_2&\rTo&y_3&    &   &\rDots&y_{q}&\rTo&x_{p+1}=y_{q+1}=z_{r+1}.\\
 &&   &    &   &    &   &      &       &\ruTo\\
   && & &z_1&\rTo&z_2&\rDots   &z_{r}\\
\end{diagram}

Consider the following vector dimension $\alpha$ of $T_{pqr}$:
$$
\alpha=\left\{
\begin{diagram}[size=2em]
   &    &a_1&\rTo&a_2&\rTo&a_3&\rDots&a_{p}\\
   &    &   &    &   &    &   &      &       &\rdTo\\
b_1&\rTo&b_2&\rTo&b_3&    &   &\rDots&b_{q}&\rTo&n.\\
 &&   &    &   &    &   &      &       &\ruTo\\
   && & &c_1&\rTo&c_2&\rDots   &c_{r}\\
\end{diagram}\right.
$$
We have the well known

\begin{lemma}\label{lem:quiverflag}
We recall that $\alpha$ is increasing on each harm.
  Then, the following are equivalent:
  \begin{enumerate}
  \item $\Fl_n(a_1,\cdots,a_{p})\times\Fl_n(b_1,\cdots,b_{q})\times\Fl_n(c_1,\cdots,c_{r})$ 
contains a dense $\GL_n$-orbit;
\item $\Rep(T_{pqr},\alpha)$ contains a dense $\GL(\alpha)$-orbit.
  \end{enumerate}
\end{lemma}

\begin{proof}
Let $R$ be a general representation of $T_{pqr}$ of dimension $\alpha$.
If $s$ is a vertex of $T_{pqr}$, $V(s)$ denotes the vector space of $R$ at $s$ and
$u(s)$ the linear map (if there exists) associated to the arrow $a$ in $T_{pqr}$
such that $ia=s$.
Since $\alpha$ is increasing on each harm, for any arrow $a$,
 the linear map $u(a)$ is injective.
In particular, the flag:
$$
\xi_x=\bigg (V(x_{p+1})\supset u(x_{p})(V(x_{p}))\supset (u(x_{p})\circ u(x_{p-1}))(V(x_{p-1}))\supset\cdots\bigg )
$$
has dimension $n> a_{p}> a_{p-1}\cdots$.
So, we obtain a point $(\xi_x,\xi_y,\xi_z)$ in 
$\Fl_n(a_1,\cdots,a_{p})\times\Fl_n(b_1,\cdots,b_{q})\times\Fl_n(c_1,\cdots,c_{r})$.
It is easy to see that $\GL(\alpha).R$ is dense in $\Rep(T_{pqr},\alpha)$ if and only if
$\GL_n.(\xi_x,\xi_y,\xi_z)$ is dense in 
$\Fl_n(a_1,\cdots,a_{p})\times\Fl_n(b_1,\cdots,b_{q})\times\Fl_n(c_1,\cdots,c_{r})$.
\end{proof}

\subsection{A Kac theorem}

We follow notation of Section~\ref{sec:defcarquois}.
We now recall a Kac's theorem which gives a criterion to decide if
$\Rep(Q,\alpha)$ contains a dense $\GL(\alpha)$-orbit.

We call $\alpha = \alpha_1+\cdots+\alpha_s$
the {\it canonical decomposition of $\alpha$} if a general
representation of dimension $\alpha$ is a direct sum of
 indecomposable representations of dimensions $\alpha_1$, $\alpha_2,\cdots,\alpha_s$.

For $\alpha,\beta\in\NN^{Q_0}$, the Ringle form is defined by:
$$
\langle\alpha,\beta\rangle=\sum_{s\in Q_0}\alpha(s)\beta(s)-
\sum_{a\in Q_1}\alpha(ia)\beta(ta).
$$

\begin{theo}(see \cite[Proposition~4]{Kac:carquois2})\label{th:Kac}
Let $\alpha=\alpha_1+\cdots+\alpha_s$ be the canonical decomposition of $\alpha$. 
Then $\Rep(Q,\alpha)$ contains a dense $\GL(\alpha)$-orbit 
if and only if for any $i=1,\cdots,s$ we have $\langle \alpha_i,\alpha_i\rangle=1$.
\end{theo}

In \cite{DW:algo}, Derksen-Weyman describe an efficient algorithm to compute the canonical decomposition
of a vector dimension. With Theorem~\ref{th:Kac}, this gives {\bf an algorithm to decide if $\Rep(Q,\alpha)$ contains a dense $\GL(\alpha)$-orbit.}
Combining this remark with Lemma~\ref{lem:quiverflag}, we obtain

\begin{prop}\label{prop:DWalgo}
Derksen-Weyman's algorithm allows to decide if the variety  
$$\Fl_n(a_1,\cdots,a_{p})\times\Fl_n(b_1,\cdots,b_{q})\times\Fl_n(c_1,\cdots,c_{r})$$
contains a dense $\GL_n$-orbit.
\end{prop}

\begin{remark}
  It would be interesting to have a classification of the  triples of parabolic subgroups $(P,Q,R)$
of $G=\GL_n$ such that $G/P\times G/Q\times G/R$ contains an open $G$-orbit; instead an algorithm to 
decide if it is. In \cite{MWZ}, Magyar-Weyman-Zelevinsky gives a classification of such triples such that
 $G/P\times G/Q\times G/R$ contains finitely many orbits. If one of $P$, $Q$, $R$ is a Borel subgroup
these two conditions are actually equivalent. Indeed, if  $G/B\times
G/Q\times G/R$ contains an open $G$-orbit,
$G/Q\times G/R$ is a spherical $G$-variety and contains by \cite{Br:BGHfini} finitely many $B$-orbits.
The case when $P=Q=R$ is maximal was obtained in \cite{Popov:gentrans}.
\end{remark}

\section{An algorithm}
\label{sec:algogen}

\subsection{Description of the algorithm}
\label{sec:algo}

In this section, we give an inductive (on $n$) algorithm to decide if
a given element of $(\Lambda_n^+)^3$ belong to 
the following subset of  ${\rm Horn}(n)$:

$$
{\rm Horn}^1(n):=\{(\lambda,\mu,\nu)\in(\Lambda_n^+)^3\,:\,c^n_{\lambda\mu\nu}=1\}.
$$

If $\lambda\in\Lambda_n^+$, we define the type of $\lambda$ by 
$$
{\rm type}(\lambda):=\{j=1,\cdots,n-1\ |\ \lambda_j\neq\lambda_{j+1}\}.
$$

 The following procedure {\tt IsLR$01$} takes an integer
 $n\geq 1$ and three elements $\lambda,\mu$ and $\nu$ in $\Lambda_n^+$
 as arguments.
It answers either $c=c^n_{\lambda\mu\nu}=0$ or $c=1$ or $c\geq 2$.

\bigskip
{\tt IsLR$01$ $(n,\lambda,\mu,\nu)$
\begin{enumerate}
\item  \begin{tabular}[t]{lll} 
If &$n=1$ then\\ 
& (if $\lambda_1+\mu_1+\nu_1=0$& 
 
    then decide $c=1$
  and stop\\
 &&else decide $c=0$ and stop)
  \end{tabular}

\item\label{casei} 
  \begin{tabular}[t]{rl}
    For every& $r=1,\cdots,n-1$
           and  $(I,\,J,\,K)$ in $\Part(r,n)$  s.t.\\&
           IsLR$01(r,\lambda^I,\lambda^J,\lambda^K-2(n-r)1^r)=1$
           do
  \end{tabular}

\begin{enumerate}
\item
Compute $\phi=\varphi_{IJK}(\lambda,\mu,\nu)$.
\item\label{caseib}
If $\phi>0$ then decide $c=0$ and stop.
\item \label{caseic}If 
  \begin{tabular}[t]{l}
$\phi=0$ then\\
Using recursively  IsLR$01$ with $r$\\ 
\ \ \ decide if $c_1=c_{\lambda_{I},\mu_{J},\nu_{K}}^r=0, 1$ or is $\geq 2$.\\
Using recursively  IsLR$01$ with $n-r$\\
\ \ \ decide if
 $c_2= c_{\lambda_{I^c},\mu_{J^c},\nu_{K^c}}^{n-r}=0, 1$ or is $\geq 2$.\\
\begin{tabular}[t]{ll}
  If $c_1=0$ or $c_2=0$ then& decide $c=0$ and stop.\\ 
If $c_1=1$ and $c_2=1$ then &decide $c=1$ and stop.\\ 
Otherwise  & decide $c\geq 2$ and stop.\\ 
\end{tabular}
\end{tabular}
\end{enumerate}
\item \label{caseii} If step \ref{casei} was inconclusive, then 
check if $$\Fl_n({\rm type}(\lambda))\times \Fl_n({\rm type}(\mu))
\times \Fl_n({\rm type}(\nu))$$ 
contains a dense $\GL_n$-orbit (using the algorithm  
of Section~\ref{sec:quiver}).\\
\begin{tabular}[l]{ll}
  If it does &then decide $c=1$ and stop \\ &else decide $c\geq 2$ and stop.
\end{tabular}
\end{enumerate}
}

\bigskip
The proof of the algorithm need some preparation.

\subsection{Modularity and GIT}

\paragraphe\label{par:nonstandardGIT}{\bf Non-standard GIT.}\ 
Let $G$ be a reductive group acting on an irreducible projective variety $X$.  
Let $\Pic^G(X)$ denote the group of $G$-linearized line bundles on $X$.
For $\Li\in\Pic^G(X)$, we denote by ${\rm H}^0(X,\Li)$ the $G$-module of regular sections of $\Li$ and
by  ${\rm H}^0(X,\Li)^G$  the subspace of $G$-invariant sections.
For any $\Li\in\Pic^G(X)$, 
we set 
$$
X^{\rm ss}(\Li)=\{x\in X \,:\,\exists n>0{\rm\ and\ }\sigma\in{\rm H}^0(X,\Li^{\otimes n})^G 
\  {\rm s.\,t.}\ \sigma(x)\neq 0\}.
$$
Note that this definition of $X^{\rm ss}(\Li)$ is as in \cite{GIT} if $\Li$ is ample
but not in general.
We consider the following projective variety:
$$
X^{\rm ss}(\Li)\quot G:={\rm Proj}\bigoplus_{n\geq 0}{\rm H}^0(X,\Li^{\otimes n})^G,
$$
and the natural $G$-invariant morphism
$$
\pi\,:\,X^{\rm ss}(\Li)\longto X^{\rm ss}(\Li)\quot G.
$$
If $\Li$ is ample $\pi$ is a good quotient. 

\bigskip\paragraphe
We denote by ${\rm mod}(X,G)$ the minimal codimension of 
$G$-orbits in $X$.  Recall that $X$ is projective, but note that the notation 
${\rm mod}(X,G)$ will be used for any irreducible $G$-variety $X$.

\begin{prop}\label{prop:modmax}
We assume that $X$ is smooth.
  The maximal of the dimensions of the varieties
$X^{\rm ss}(\Li)\quot G$ for $\Li\in\Pic^G(X)$ is equal to
  ${\rm mod}(X,G)$.
\end{prop}

\begin{proof}
  Let  $\Li\in\Pic^G(X)$.
Since $\pi$ is $G$-invariant, we have:
$$
\dim (X^{\rm ss}(\Li)\quot G)\leq{\rm mod}(X^{\rm ss}(\Li),G)={\rm mod}(X,G).
$$
Conversely, set $m={\rm mod}(X,G)$. It remains to construct $\Li$ such that 
$\dim(X^{\rm ss}(\Li)\quot G)\geq m$.
It is well known that $m$ is the transcendence degree of the field
$\CC(X)^G$ of $G$-invariant rational functions on $X$.
Let $f_1,\cdots,f_m$ be algebraically independent elements of $\CC(X)^G$.
For each $i=1,\cdots,m$, consider the two effective divisors $D_i^0$ and $D_i^\infty$ such that 
$\div(f_i)=D_i^0-D_i^\infty$.
Consider the line bundle $\Li_i={\mathcal O}(D_i^0)={\mathcal O}(D_i^\infty)$.
Let $\sigma_i^0$ be a regular section of $\Li_i$ such that $\div(\sigma_i^0)=D_i^0$.
Since $D_i^0$ is $G$-stable, there exists a unique $G$-linearization of  $\Li_i$ such that
$\sigma_i^0$ is $G$-invariant; we now consider $\Li_i$ endowed with this linearization.
There exists a unique section $\sigma_i^\infty$ of $\Li_i$ such that $f_i=\sigma_i^0/\sigma_i^\infty$;
since $f_i$ and $\sigma_i^0$ are $G$-invariant, so is $\sigma_i^\infty$. 

Set $\Li=\Li_1\otimes\cdots\otimes\Li_m$. Consider the following
$G$-invariant sections of $\Li$:
$$
\begin{array}{l}
\tau_i=\sigma_1^\infty\otimes\cdots\otimes\sigma_{i-1}^\infty\otimes\sigma_{i}^0\otimes
\sigma_{i+1}^\infty\otimes\cdots\otimes\sigma_m^\infty\ \ \ \forall i=1,\cdots,m,\\
\tau_0=\sigma_1^\infty\otimes\cdots\otimes\sigma_{m}^\infty.
\end{array}
$$
Consider now the rational map
$$
\begin{array}{rcl}
  \theta\,:\,X^{\rm ss}(\Li)&\longratto&\PP^m\\
x&\longmapsto&[\tau_0(x):\cdots:\tau_m(x)].
\end{array}
$$
Since $f_1,\cdots,f_m$ are algebraically independent, $\theta$ is
dominant. 
Since $\theta$ is defined by $G$-invariant sections of $\Li$, it  factors
by $\pi\,:\,X^{\rm ss}(\Li)\longto X^{\rm ss}(\Li)\quot G$. It follows that 
$\dim(X^{\rm ss}(\Li)\quot G)\geq m$.
\end{proof}

\bigskip\paragraphe
We assume here that $\Pic^G(X)$ has finite rank and consider the rational vector
space $\Pic^G(X)_\QQ:=\Pic^G(X)\otimes_\ZZ\QQ$.
Since $X^{\rm ss}(\Li)=X^{\rm ss}(\Li^{\otimes n})$ for any positive integer $n$, 
one can define $X^{\rm ss}(\Li)$ for any element in  $\Pic^G(X)_\QQ$.
The set of ample line bundles in $\Pic^G(X)$ generates an open convex cone 
$\Pic^G(X)_\QQ^+$ in  $\Pic^G(X)_\QQ$.  The following cone was defined in 
\cite{DH} and will be called the {\it ample GIT-cone}:

$$
\ac^G(X):=\{\Li\in\Pic^G(X)_\QQ^+\,:\,X^{\rm ss}(\Li)\neq\emptyset\}.
$$
Indeed, since the product of two nonzero $G$-invariant sections of two line 
bundles is a nonzero $G$-invariant section of the tensor product of the two line bundles, 
$\ac^G(X)$ is convex.
The following result is certainly well-known and can be deduced from 
\cite[Proposition~1.1]{GITEigen2}:

\begin{prop}
\label{prop:dimquotconst}
The dimension of $X^{\rm ss}(\Li)\quot G$ does not depend on $\Li$ in the relative 
interior of $\ac^G(X)$.
\end{prop}

\bigskip\paragraphe
We now consider the case when $X$ is a product of flag manifolds:

\begin{lemma}\label{lem:modularityflag}
We assume that $X$ is a product of flag manifolds for $G$ and that  $\ac^G(X)$ is nonempty.
For any $\Li$ in the relative interior of $\ac^G(X)$, the dimension of $X^{\rm ss}(\Li)\quot G$
equals ${\rm mod}(X,G)$.
\end{lemma}

\begin{proof}
The lemma is a consequence of Propositions~\ref{prop:dimquotconst} and \ref{prop:modmax}.
The only difficulty is that Proposition~\ref{prop:dimquotconst} deals with ample line bundles 
and Proposition~\ref{prop:modmax} concerns any line bundle.

  Let $\Mi\in \Pic^G(X)$ such that $X^{\rm ss}(\Mi)$ is not empty.
By \cite[Proposition~10]{GITEigen}, $\Mi$ belongs to the closure of   $\ac^G(X)$.
By \cite[Lemma~7]{GITEigen}, there exists $\Li$ in the relative interior of  $\ac^G(X)$
such that $X^{\rm ss}(\Li)\subset X^{\rm ss}(\Mi)$. Corresponding to this inclusion we have a
dominant (and so surjective) morphism  $X^{\rm ss}(\Li)\quot G\longto X^{\rm ss}(\Mi)\quot G$.
In particular, we have:
$$
\dim(X^{\rm ss}(\Li)\quot G)\geq \dim(X^{\rm ss}(\Mi)\quot G).
$$
With Proposition~\ref{prop:dimquotconst}, this implies that for any $\Li$ in the relative 
interior of $\ac^G(X)$, the dimension of $X^{\rm ss}(\Li)\quot G$ equals the maximal 
dimension of the varieties $X^{\rm ss}(\Mi)\quot G$ for $\Mi\in\Pic^G(X)$.
With Proposition~\ref{prop:modmax}, this implies the lemma.
\end{proof}

\subsection{Properties of the LR-coefficients}

\paragraphe{\bf Saturation.}
Let $(\lambda,\mu,\nu)\in\Lambda^+_n$ and $m$ be a positive integer. 
Knutson-Tao proved in \cite{KT:saturation}:
\begin{theo}\label{th:saturation}
If 
$c^n_{m\lambda\; m\mu\; m\nu}\neq 0$ then
$c^n_{\lambda \mu\nu}\neq 0$.
\end{theo}

A geometric proof is given in \cite{Bel:saturation}.
Note that this statement is a corollary (or a part) of
Theorem~\ref{th:Horn} and was already stated in the introduction.

\bigskip\paragraphe{\bf The Fulton conjecture.}
Let $(\lambda,\mu,\nu)\in\Lambda^+_n$ and $m$ be a positive integer.
 Knutson-Tao proved in \cite{KT:saturation} the following Fulton conjecture:
 \begin{theo}
   \label{th:Fultonconj}
If $c^n_{\lambda\mu\nu}=1$ then $c^n_{m\lambda\; m\mu\;m\nu}=1$ for
any positive integer $m$.
 \end{theo}

Geometric proofs of this result are given in
\cite{belkale:geomHorn,fulton,BKR}.
 
\bigskip\paragraphe{\bf LR-coefficients on the boundary of ${\rm Horn}(n)$.}
The following theorem has been proved independently in \cite{KTT:factorLR} and \cite{DW:comb}.
Alternative proofs can be found in \cite{Roth:red,ReductionRule}.

\begin{theo}\label{th:prod}
  Let $(\lambda,\mu,\nu)\in(\Lambda_n^+)^3$. Let
  $(I,J,K)\in \Part(r,n)$ such that 
$$
\sigma_I.\sigma_J.\sigma_K=[{\rm pt}].
$$
If $\varphi_{IJK}(\lambda,\mu,\nu)=0$, then
$$
c^n_{\lambda\,\mu\,\nu}=c^r_{\lambda_I\,\mu_J\,\nu_K}\;.\; 
c^{n-r}_{\lambda_{I^c}\,\mu_{J^c}\,\nu_{K^c}}.
$$
\end{theo}
\subsection{Proof of the algorithm}

\begin{theo}
The algorithm described in Section~\ref{sec:algo} decides if
$c^n_{\lambda\mu\nu}=0$, $1$ or if $c^n_{\lambda\mu\nu}\geq 2$.
\end{theo}

\begin{proof}
The case $n=1$ is obvious. Moreover, the procedure is used recursively
three times with strictly smaller $n$. So, the procedure finishes.

If algorithm stop in case~\ref{caseib}, we have 
$$
\varphi_{IJK}(\lambda,\mu,\nu)>0
$$
for some $(I,J,K)$ appearing in  Theorem~\ref{th:Belk}.
This implies that  $c^n_{\lambda\mu\nu} =0$. 

If the algorithm stop in case~\ref{caseic}, $\phi$ is equal to $0$.
Then  Theorem~\ref{th:prod} shows that $c=c_1.c_2$; and
the algorithm works in this case.

We now consider case~\ref{caseii}. In this case, 
for any $r=1,\cdots,n-1$ and for any $(I,J,K)\in \Part(r,n)$  such
that $\sigma_I.\sigma_J.\sigma_K=[{\rm pt}]$, we have:
$$
\varphi_{IJK}(\lambda,\mu,\nu)<0.
$$
So, Theorem \ref{th:Belk} shows that $(\lambda,\mu,\nu)\in {\rm Horn}_\RR(n)$.

Let $T$ and $B$ be the usual maximal torus and Borel subgroup of $\GL_n$.
Then, $\lambda$ corresponds to a character of $T$ or $B$.
The group $B$ fixes a unique point in $\Fl_n({\rm type}(\lambda))$ whose the
stabilizer in $G$ will be denoted by $P$.
Moreover, $\lambda$ extends to unique character of $P$.
Similarly, we can think about $\mu$ and $\nu$ as characters of
parabolic subgroups $Q$ and $R$.
Consider the $G=\GL_n^3$-variety $X=\Fl_n({\rm type}(\lambda))\times \Fl_n({\rm type}(\mu))
\times \Fl_n({\rm type}(\nu))=G/P\times G/Q\times G/R$.
Let $\Li$ be the $\GL_n^3$-linearized line bundle on $X$ associated to $(\lambda,\mu,\nu)$ 
(see Paragraph~\ref{par:Lnu} below for details).
It is well known that $\Li$ is ample and that 
$H^{\rm 0}(X,\Li^{\otimes m})=V^*_{m\lambda}\otimes V^*_{m\mu}\otimes V^*_{m\nu}$,
for any positive integer $m$.

Let $\bar \Li$ be the $\GL_n$-linearized line bundle on $X$ obtained by restriction the action of
$\GL_n^3$ to the diagonal.
Since each $\varphi_{IJK}(\lambda,\mu,\nu)<0$, Theorem~\ref{th:Belk} implies that
$\bar \Li$ belongs to the relative interior of $\ac^{G}(X)$. 
Now, Lemma~\ref{lem:modularityflag} implies that the dimension of 
$X^{\rm ss}(\bar \Li)\quot G$ is ${\rm mod}(X,G)$.

Assume now that Derksen-Weyman's algorithm decides that  $X$ does not
contain an open $G$-orbit; that is ${\rm mod}(X,G)>0$.
Since the dimension of $X^{\rm ss}(\bar \Li)\quot G$ is positive,
there exists a positive integer such that
$c_{m\lambda\,m\mu\,m\nu}^n\geq 2$.
Now, Fulton's conjecture implies that $c_{\lambda\mu\nu}^n\neq 1$.
But, Theorem \ref{th:saturation} implies that
$c_{\lambda\mu\nu}^n\neq 0$. Finally,  $c_{\lambda\mu\nu}^n\geq 2$.

Assume finally that ${\rm mod}(X,G)=0$.
Since $X^{\rm ss}(\bar \Li)\quot G$ is a point,
$c_{\lambda\mu\nu}^n\leq 1$.
But, Theorem \ref{th:saturation} implies that
$c_{\lambda\mu\nu}^n\neq 0$. Finally,  $c_{\lambda\mu\nu}^n=1$.
\end{proof}

\section{A parametrization of Schubert varieties}
\label{sec:schub}

In this section, we recall some properties about the inversion sets introduced by Kostant 
in \cite{Kostant:harmform1}.

\subsection{The general case}
\label{sec:paramschubert}

\paragraphe
Let $G$ be a complex reductive group.
Let $T\subset B$ be a maximal torus and a Borel subgroup of $G$.

Let $\Phi$ (resp. $\Phi^+$) denote the set of roots (resp. positive roots) of $G$.
Set $\Phi^-=-\Phi^+$.
Let $\Delta$ denote the set of simple roots.
Let us consider the set $X(T)^+$ of dominant characters of $T$.
Let $W$ denote its Weyl group.

\bigskip\paragraphe\label{par:Lambda}
Let $P$ be a standard (ie which contains $B$) parabolic subgroup of $G$
and $L$ denote its Levi subgroup containing $T$.
Let $W_L$ denote the Weyl group of $L$ and $\Phi_L$ denote the set of roots of $L$.
We consider the homogeneous space $G/P$. Its base point is denoted by $P$.

 For $w\in W/W_L$, we consider the associated Schubert variety 
$\Omega(w)$ which is the closure of $BwP/P$. 

If $G/P$ is a Grassmannian, the Schubert varieties are classically parametrized by partitions
(see Paragraphs~\ref{par:SchubGr} and \ref{par:lambdaI}).
We are going to generalize this parametrization.
The set of weights of $T$ acting on the tangent space $T_PG/P$ is $-(\Phi^+\backslash \Phi_L)$.
Set 
$$\Lambda(G/P)=-(\Phi^+\backslash \Phi_L).$$ 

Let $W^P$ denote the set of minimal length representatives of elements in $W/W_L$.
Let $w\in W^P$.
Consider $w^{-1}\Omega(w)$: it is a closed $T$-stable subvariety of $G/P$ containing $P$ and smooth at $P$.
The tangent space $T_{P}w^{-1}\Omega(w)$ is called {\it the centered tangent space of $\Omega(w)$}.
We set:
$$
\Lambda_w=\{\alpha\in\Lambda(G/P)\,:\, \alpha{\rm\ is\ not\ a\ weight\ of\ }T{\rm\ in\ }T_Pw^{-1}\Omega(w)\}.
$$
Let $\Part(\Lambda(G/P))$ denote the set of parts of $\Lambda(G/P)$. 
We have the following easy lemma (see~\cite{Bou}).

\begin{lemma}
\label{lem:partgen}
We have $\Lambda_w=\{\alpha\in\Lambda(G/P)\,:\,-w\alpha\in\Phi^+\}$, and
the map $W^P\longto \Part(\Lambda(G/P))$, $w\mapsto\Lambda_w$ is injective.
Moreover, the codimension of $\Omega(w)$ is the cardinality of $\Lambda_w$.   
\end{lemma}

\paragraphe
We write $\alpha\lleq \beta$ if $\beta-\alpha$ is a non-negative combination of 
positive roots.

If $\lambda$ is a one parameter subgroup of $G$ then  the set of $g\in G$ such 
that $\lim_{t\to 0}\lambda(t)g\lambda(t^{-1})$ exists in $G$ is
a parabolic subgroup $P(\lambda)$ of $G$.
Moreover, any parabolic subgroup of $G$  can be obtained
in such a way.
Let us fix a one parameter subgroup $\lambda$ of $T$ such that $P=P(\lambda)$.
Let $\langle \cdot\,,\,\cdot\rangle$ denote the natural paring between one parameter subgroups and 
characters of $T$.

\begin{lemma}
\label{lem:orderideal}
  Let $\alpha\in\Lambda_w$ and $\beta\in\Lambda(G/P)$. We assume that
$\langle\lambda,\alpha\rangle=\langle\lambda,\beta\rangle$ and
$\beta\lleq\alpha$.

Then, $\beta\in\Lambda_w$.
\end{lemma}

\begin{proof}
  We have to prove that $w\beta\in\Phi^-$. 
But $w\beta=w\alpha+w(\beta-\alpha)$. 
Since $\langle\lambda,\beta-\alpha\rangle=0$, $\beta-\alpha$ belongs to the root lattice of $L$.
But, $\beta\lleq\alpha$; so, $\beta-\alpha$ is a non-negative combination of negative roots of $L$.
Since $w\in W^P$,  $w\Phi_L^-\subset\Phi^-$.
Finally, $  w(\beta-\alpha)$ is a  non-negative combination of negative roots.
If follows that $w\beta\lleq w\alpha$ and $w\beta\in\Phi^-$.
\end{proof}

\bigskip

Lemma~\ref{lem:orderideal} implies that $\Lambda_w$ is an order ideal on each strata given 
by $\lambda$. More precisely, Kostant characterized \cite[Proposition 5.10]{Kostant:harmform1} the parts of $\Lambda(G/P)$
equal to $\Lambda_w$ for some $w\in W^P$.

\subsection{The case $\SL_n$}
\paragraphe\label{par:defrootSL}
Let $V$ be a $n$-dimensional vector space and set $G=\SL(V)$.
Let $\base=(e_1,\cdots,e_{n})$ be a basis of $V$. Let $T$ be the maximal torus
of $G$ consisting of diagonal matrices in $\base$ and $B$ the Borel subgroup
of $G$ consisting of upper triangular matrices.
Let $\varepsilon_i$ denote the character of $T$ which maps 
$\diag(t_1,\cdots,t_{n})$ to $t_i$; we have 
$X(T)=\oplus_i\ZZ\varepsilon_i/\ZZ\sum_i\varepsilon_i$.
Here, we have:
$$
\begin{array}{l}
  \Phi^+=\{\varepsilon_i-\varepsilon_j\,:\,i<j\},\\
\Delta=\{\alpha_r=\varepsilon_r-\varepsilon_{r+1}\,:\,r=1,\cdots,n-1\}.
\end{array}
$$
The Weyl group $W$ of $G$ is the symmetric group $S_{n}$ acting  on $n$ letters.
We will denote by $F(r)$ the span of $e_1,\cdots,e_r$.\\

\paragraphe
Let $\alpha_r$ be a simple root, $P_r$ be the corresponding maximal standard parabolic subgroup of $G$ 
and $L_r$ be its Levi subgroup containing $T$.
The homogeneous space $G/P_r$ with base point $P_r$ is the Grassmannian $\Gr(r,n)$ of $r$-dimensional 
subspaces of $V$ with base point $F(r)$.
The tangent space $T_{F(r)}\Gr(r,n)$ identifies with ${\rm Hom}(F(r),V/F(r))$.
The natural action of $L_r$ which is isomorphic to ${\rm S}(\GL(F(r))\times\GL(V/F(r)))$ makes this
identification equivariant.

Consider $\Lambda(\Gr(r,n))=\Phi^-\backslash\Phi_{L_r}$ as in Paragraph~\ref{par:Lambda}:
$$
\Lambda(\Gr(r,n))=\{\varepsilon_i-\varepsilon_j\,:\, 1\leq j\leq r<i\leq n\}.
$$
We now represent $\Lambda(\Gr(r,n))$ by a rectangle with $r\times (n-r)$ boxes: the box at row $i$ and
the column $j$ represents the root $\varepsilon_{r+i}-\varepsilon_j$ (see Figure~\ref{fig:Gr1}).

\begin{figure}
\centering\psset{unit=0.5cm}
\begin{pspicture}(0,0)(6,5)
\multido{\n=0+1}{5}{
\psline(0,\n)(6,\n)
}
\uput[r](6,3.5){$\varepsilon_{r+1}$}
\uput[r](6,2){$\vdots$}
\uput[r](6,0.5){$\varepsilon_n$}
\multido{\n=0+1}{7}{
\psline(\n,0)(\n,4)
}  
\uput[u](0.5,4){$\varepsilon_{1}$}
\uput[u](3,4){$\dots$}
\uput[u](5.5,4){$\varepsilon_{r}$}
\end{pspicture}
\caption{$\Lambda(\Gr(r,n))$}\label{fig:Gr1}
\end{figure}

Note that Lemma~\ref{lem:orderideal} asserts in this case that 
the $\Lambda_w$'s are Young diagrams (oriented as  on Figure~\ref{fig:Gr2}).

\bigskip\paragraphe\label{sec:lambdaI}
If $I\in\Part(r,n)$, we set $F(I)={\rm Span}(e_i\,:\,i\in I)$.
Let $I=\{i_1<\cdots<i_r\}$ and $\Omega(I)$ the corresponding Schubert variety, that is the 
closure of $B.F(I)$. Set $\{i_{r+1}<\cdots<i_{n}\}=I^c$.
Set $w_I=(i_1,\cdots,i_{n})\in S_{n}=W$; 
then, $w_I\in W^{P_r}$
and represents $\Omega(I)$. Set $\Lambda_I=\Lambda_{w_I}$; we have:
$$
\Lambda_I=\{\varepsilon_i-\varepsilon_j\,:\, w_I(j)<w_I(i) {\rm \ and\ }j\leq r<i\}.
$$

To obtain $\Lambda_I$ on Figure~\ref{fig:Gr2}, one can proceeds as follows.
Index the columns (resp. rows) of Figure~\ref{fig:Gr1} by $I$ (resp. $I^c$).
Now, a given box belongs to $\Lambda_I$ if and only if the index of its column is less that those of 
its row. 
For example, if $I=\{1,\,4,\,5,\,7,\,8,\,10\}\in\Part(6,10)$, 
$\Lambda_I$ is the set of black boxes on Figure~\ref{fig:Gr2}.

\begin{figure} \centering
  \psset{unit=5mm}
\begin{pspicture}(0,0)(6,5)
\pspolygon[fillstyle=solid,fillcolor=gray](0,0)(0,4)(1,4)(1,2)(3,2)(3,1)(5,1)(5,0)
\multido{\n=0+1}{5}{
\psline(0,\n)(6,\n)
}
\multido{\n=0+1}{7}{
\psline(\n,0)(\n,4)
}  

\uput[r](6,3.5){$2$}
\uput[r](6,2.5){$3$}
\uput[r](6,1.5){$6$}
\uput[r](6,0.5){$9$}

\uput[u](0.5,4){$1$}
\uput[u](1.5,4){$4$}
\uput[u](2.5,4){$5$}
\uput[u](3.5,4){$7$}
\uput[u](4.5,4){$8$}
\uput[u](5.5,4){$10$}
\end{pspicture}
  \caption{An example of $\Lambda_I$}
  \label{fig:Gr2}
\end{figure}
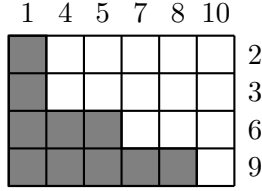

Note that $\Lambda_I$ is the {\bf complement of the transpose of} the
Young diagram of $\lambda^I$ as defined 
in Paragraph~\ref{par:lambdaI}.

\bigskip\paragraphe\label{par:Lambdap}
We now consider the case of a two step flag manifold
$
\Fl_n(r_1,r_2).
$
Here, $\Lambda(\Fl_n(r_1,r_2))$ is the union of three rectangles of size 
$r_1\times (r_2-r_1)$, $(r_2-r_1)\times (n-r_2)$ and $r_1\times (n-r_2)$
(see Figure~\ref{fig:Fl1}). These three rectangles are denoted by
$R_0$, $R_1$ and $R_2$ respectively.

\begin{figure}
  \centering

\psset{unit=0.5cm}
\begin{pspicture}(-3,-1.3)(7,8)
\multido{\n=0+1}{4}{
\psline(0,\n)(6,\n)
}\multido{\n=4+1}{4}{
\psline(0,\n)(2,\n)
}
\uput[r](6,6.5){$\varepsilon_{r_1+1}$}
\uput[r](6,3.5){$\vdots$}
\uput[r](6,0.5){$\varepsilon_n$}

\multido{\n=0+1}{3}{
\psline(\n,0)(\n,7)
} \multido{\n=3+1}{4}{
\psline(\n,0)(\n,3)
}
  
\uput[u](0.5,7){$\varepsilon_{1}$}
\uput[u](1.5,7){$\varepsilon_{r_1}$}
\uput[u](3,7){$\dots$}
\uput[u](5.5,7){$\varepsilon_{r_2}$}

\psline{<->}(-0.3,3)(-0.3,7)\psline{<->}(-0.3,3)(-0.3,0)
\psline{<->}(0,-0.3)(2,-0.3)\psline{<->}(2,-0.3)(6,-0.3)
\uput[l](-0.3,5){$r_2-r_1$}\uput[l](-0.3,1.5){$n-r_2$}
\uput[d](1,-0.3){$r_1$}\uput[d](4,-0.3){$r_2-r_1$}

\psline[linewidth=2pt](0,3)(2,3)(2,0)
\end{pspicture}
  
  \caption{$\Lambda(\Fl_n(r_1,r_2))$}
  \label{fig:Fl1}
\end{figure}

The Schubert varieties are naturally parametrized by the set $\Schub(\Fl_n(r_1,r_2))$ of the
pairs $(I^1,\,I^2)\in\Part(r_1,n)\times\Part(r_2,n)$ such that $I^1\subset I^2$.
Let $(I^1,I^2)\in\Schub(\Fl_n(r_1,r_2))$.
To obtain $\Lambda_p$ on Figure~\ref{fig:Fl1}, one can proceed as follows.
Index the $r_1$ first columns (resp. $r_2-r_1$ first rows) of Figure~\ref{fig:Fl1} by $I^1$ 
(resp. $I^2-I^1$).
Index the following $r_2-r_1$ columns (resp. $n-r_2$ rows) of Figure~\ref{fig:Fl1} by $I^2-I^1$ 
(resp. $[1,n]-I^2$).
Now, a given box belongs to $\Lambda_{(I^1,I^2)}$ if and only if the index of its column is less that those of 
its row. 
For example, if $n=9$, $I^1=\{3,\,7\}$ and $I^2=I^1\cup\{1,\,5,\,6,\,8\}$, 
one obtains $\Lambda_{(I^1,I^2)}$ on Figure~\ref{fig:Fl2}.

\begin{figure}
  \centering
\psset{unit=0.5cm}
\begin{pspicture}(-1,-1)(7.5,7) 
\uput[d](1,0){$\Lambda_{2}$}
\uput[r](2,5){$\Lambda_{0}$}
\uput[d](4,0){$\Lambda_{1}$}

\pspolygon[fillstyle=solid,fillcolor=gray](0,3)(0,6)(1,6)(1,4)(2,4)(2,3)
\pspolygon[fillstyle=solid,fillcolor=gray](0,0)(0,2)(1,2)(1,1)(2,1)(2,0)
\pspolygon[fillstyle=solid,fillcolor=gray](2,0)(2,3)(3,3)(3,1)(6,1)(6,0)

\multido{\n=0+1}{4}{
\psline(0,\n)(6,\n)
}\multido{\n=4+1}{4}{
\psline(0,\n)(2,\n)
}
\multido{\n=0+1}{3}{
\psline(\n,0)(\n,7)
} \multido{\n=3+1}{4}{
\psline(\n,0)(\n,3)
}

\uput[u](0.5,7){$3$}\uput[u](1.5,7){$7$}
\uput[u](2.5,3){$1$}\uput[u](3.5,3){$5$}\uput[u](4.5,3){$6$}\uput[u](5.5,3){$8$}
\uput[l](0,6.5){$1$}\uput[l](0,5.5){$5$}\uput[l](0,4.5){$6$}\uput[l](0,3.5){$8$}
\uput[l](0,2.5){$2$}\uput[l](0,1.5){$4$}\uput[l](0,0.5){$9$}

\psline[linewidth=2pt](0,3)(2,3)(2,0)
\end{pspicture}
  \caption{An example of $\Lambda_{(I^1,I^2)}$ for ${(I^1,I^2)}\in\Schub(\Fl_n(r_1,r_2))$}
  \label{fig:Fl2}
\end{figure}

\begin{remark}
Lemma~\ref{lem:orderideal} means that $\Lambda_{(I^1,I^2)}$ is the union of three Young diagrams 
as on Figure~\ref{fig:Fl2}.
\end{remark}

\bigskip\paragraphe
\label{par:SFl}
We now consider the following 
characteristic function:
$$
\begin{array}{llll}
\chi_{(I^1,I^2)}\,:&[1;n]&\longto&\{0,1,2\}\\
         &  i   &\longmapsto&\left\{
           \begin{array}{rll}
             1&{\rm\ if\ }&i\in I^1,\\
             2&{\rm\ if\ }&i\in I^2-I^1,\\
             0&{\rm\ if\ }&i\not\in I^2.\\
           \end{array}
\right.  
\end{array}
$$
We think about  $\chi_{(I^1,I^2)}$ as a word of length $n$ with letters in 
$\{0,1,2\}$.
If one cancels the letters $2$ in this word, one obtains the characteristic function of 
a part $I_2$ of $[1;n-(r_2-r_1)]$ with $r_1$ elements.
 If one cancels the letters $1$ in this word and then replaces $2$ by $1$, one obtains the 
characteristic function of a part $I_{1}$ of $[1;n-r_1]$ with $r_2-r_1$ elements.
 If one cancels the letters $0$ in this word and then replaces $2$ by $0$, one obtains the 
characteristic function of a part $I_{0}$ of $[1;r_2]$ with $r_1$ elements.
We just defined a map:
\begin{eqnarray}
  \label{eq:2}
\begin{array}{cll}
  \Schub(\Fl_n(r_1,r_2))&\longto&
\Part(r_1,n+r_1-r_2)\times \Part(r_2-r_1,n-r_1)\times \Part(r_1,r_2)\\
   (I^1,I^2)   &\longmapsto&(I_2,I_1,I_0).
\end{array}
\end{eqnarray}

Let $ (I^1,I^2)\in \Schub(\Fl_n(r_1,r_2))$. Set $\Lambda_i=\Lambda_{ (I^1,I^2)}\cap R_i$. 

\begin{prop}
  With above notation, $\Lambda_i$ is the partition associated to the part $I_i$, for 
$i=0,\,1$ and $2$.
\end{prop}

\begin{proof}
  The proof is direct  with the description of $\Lambda_{ (I^1,I^2)}$ made in Paragraph~\ref{par:Lambdap}.
\end{proof}

\bigskip\paragraphe
We now consider the particular case when 
$n-r_2=r_1$. So consider $\Fl_n(r,n-r)$. In this case $\Lambda(G/P)$ 
is symmetric under the diagonal dashed line on Figure~\ref{fig:GammaSp} below.
Let $\tau$ denote this symmetry.

For $i\in [1;n]$, we set $\overline{i}=n+1-i$. The symmetry $\tau$
corresponds to the involution $\overline{\square}$. More precisely, we have:

\begin{lemma}
\label{lem:sym}
Let $(I^1,I^2)\in\Schub(\Fl_n(r,n-r))$.
Set $J^1=I^1$, $J^2=I^2-I^1$ and $J^3=[1;n]-J^2$.

Consider $(\overline{J^3},\overline{J^2},\overline{J^1})$; 
and $(\overline{J^3},\overline{J^2}\cup  \overline{J^3})\in\Schub(\Fl_n(r,n-r))$.

Then, $\tau(\Lambda_{ (I^1,I^2)})=\Lambda_{(\overline{J^3},\overline{J^2}\cup  \overline{J^3})}$.
\end{lemma}

\begin{proof}
  The proof is direct with the description of $\Lambda_{ (I^1,I^2)}$ made in Paragraph~\ref{par:Lambdap}.
\end{proof}

\subsection{The case $\Sp_{2n}$}
\label{sec:expleSp}

\paragraphe{\bf Root system.}
\label{par:rsSp}

Let $V$ be a $2n$-dimensional vector space 
and $\base=(e_1,\cdots,e_{2n})$ be a basis of $V$.
Consider the following $n\times n$ matrix $J_n$:
\begin{eqnarray}
  \label{eq:defJn}
J_n=\left(
  \begin{array}{ccc}
    &&1\\
&\revddots\\
1
  \end{array}
\right);
\end{eqnarray}

and the bilinear symplectic form $\omega$ on 
$V$ with matrix
$$
\omega=\left(
  \begin{array}{cc}
   0 &J_n\\
-J_n&0
  \end{array}
\right).
$$ 
Let $G$ be the associated symplectic group.
Set $T=\{\diag(t_1,\cdots,t_n,t_n^{-1},\cdots,t_1^{-1})\,:\,t_i\in\CC^*\}$.
Let $B$ be the Borel subgroup of $G$ consisting of upper triangular matrices of $G$.
For $i\in [1,n]$, let $\varepsilon_i$ denote the character of $T$ which maps 
$\diag(t_1,\cdots,t_n,t_n^{-1},\cdots,t_1^{-1})$ to $t_i$; we have 
$X(T)=\oplus_i\ZZ\varepsilon_i$.
Here, we have:
$$
\begin{array}{l}
  \Phi^+=\{\varepsilon_i\pm\varepsilon_j\,:\,1\leq i<j\leq n\}\cup 
\{2\varepsilon_i\,:\,1\leq i\leq n\},\\
\Delta=\{\alpha_1=\varepsilon_1-\varepsilon_2,\,\alpha_2=\varepsilon_2-\varepsilon_3,\cdots,\,
\alpha_{n-1}=\varepsilon_{n-1}-\varepsilon_n,\,\alpha_n=2\varepsilon_n\}.
\end{array}
$$
If $i\in [1;2n]$, we set $\overline{i}=2n+1-i$.
The Weyl group $W$ of $G$ is a subgroup of the Weyl group $S_{2n}$ of $\SL(V)$:
$$
W=\{w\in S_{2n}\,:\,w(\overline{i})=\overline{w(i)} \ \ \forall i\in [1;2n]\}.
$$
We will denote by $F(r)$ (resp. $\overline{F}(r)$) the span of $e_1,\cdots,e_r$ 
(resp. $e_{\overline{1}},\cdots,e_{\overline{r}}$).
We will denote by $V(r)$ the span of $e_{r+1},\cdots,e_{\overline{r+1}}$.
\\

\paragraphe{\bf Tangent space of isotropic Grassmanians.}
Let $\alpha_r$ be a simple root, $P_r$ be the corresponding maximal
standard parabolic 
subgroup of $G$ and $L_r$ be its Levi subgroup containing $T$.
The homogeneous space $G/P_r$ with base point $P_r$ is the isotropic Grassmannian $\Gr_\omega(r,2n)$ of 
$r$-dimensional subspaces $M$ of $V$ such that $\omega(M,M)=0$ with base point $F(r)$.

Note that $V=F(r)\oplus V(r)\oplus\overline{F}(r)$.
Moreover, $F(r)^{\perp_\omega}=F(r)\oplus V(r)$, and $\omega$ identifies $\overline{F}(r)$ 
with the dual of $F(r)$.
The tangent space $T_{F(r)}\Gr_\omega(r,2n)$ identifies with ${\rm Hom}(F(r),V(r))\oplus S^2F(r)^*$.
The natural action of $L_r$ which is isomorphic to $\GL(F(r))\times\Sp(V(r))$ makes this
identification equivariant.

For convenience we set for $i=1,\cdots,n$, $\varepsilon_{\overline{i}}:=-\varepsilon_i$.
Then, 
$$
\begin{array}{r@{\:}c@{\:}l@{\,}l@{\,}l}
\Phi^-&=&\{\varepsilon_i-\varepsilon_j\,:\, 1\leq j&                           <i\leq\overline{j}\leq 2n\},&{\rm\, and}\\ 
  \Lambda(\Gr_\omega(r,2n))&=&\{\varepsilon_i-\varepsilon_j\,:\, 1\leq j\leq r&<i\leq\overline{j}\leq 2n\}.
\end{array}
$$
We now represent each element of $\Lambda(\Gr_\omega(r,2n))$ by a box on Figure~\ref{fig:GammaSp}.
The box at row $i$ and column $j$ corresponds to $\varepsilon_{r+i}-\varepsilon_j$.

\begin{figure}\psset{unit=5mm}
\centering
\begin{pspicture}(0,0)(4,7)
\multido{\n=3+1}{5}{
\psline(0,\n)(4,\n)
}
\multido{\n=0+1,\i=1+1}{3}{
\psline(0,\n)(\i,\n)
}
\psline(0,0)(0,7)
 
\multido{\i=1+1,\n=0+1}{4}{
\psline(\i,\n)(\i,7)
} 

\uput[u](0.5,7){$\varepsilon_1$}
\uput[u](2,7){$\dots$}
\uput[u](3.5,7){$\varepsilon_r$}
\uput[l](0,6.5){$\varepsilon_{r+1}$}
\uput[l](0,3.5){$\vdots$}
\uput[l](0,0.5){$\varepsilon_{2n}$}

\psline[linestyle=dashed](0,0)(7,7)
\psline[linestyle=dotted](0,0)(7,0)(7,4)(4,4)

\psline[linewidth=2pt](0,4)(4,4)
\end{pspicture}
\caption{Roots of $T_{F(r)}\Gr_\omega(r,2n)$}
\label{fig:GammaSp}
\end{figure}

The boxes corresponding to roots of $S^2F(r)^*$ (resp. $\Hom(F(r),V(r))$) are in the triangular 
(resp. rectangular) part of Figure~\ref{fig:GammaSp}.\\
 
\paragraphe\label{par:wISp}
{\bf Schubert varieties of isotropic Grassmanians.}
If $I\in\Part(r,2n)$ then we set $\overline{I}=\{\overline{i}\,:\,i\in I\}$ and
$$
\Schub(\Gr_\omega(r,2n)):=\{I\in\Part(r,2n)\,:\, I\cap\overline{I}=\emptyset\}.
$$
The subspace  $F(I)$ belongs to $\Gr_\omega(r,2n)$ if and only if 
$I\in \Schub(\Gr_\omega(r,2n))$; so, the Schubert varieties $\Psi(I)$ of $\Gr_\omega(r,2n)$ are indexed  
by $I\in\Sch(\Gr_\omega(r,2n))$.
If $I=\{i_1<\cdots<i_r\}\in \Sch(r,2n)$, we set $i_{\overline{k}}=\overline{i_k}$ and write 
$(I\cup\overline{I})^c=\{i_{r+1}<\cdots<i_{\overline{r+1}}\}$. 
Then, the element of $W^{P_r}$ which corresponds to $\Psi(I)$ is
$w_I=(i_1,\cdots,i_{2n})$.\\

\paragraphe\label{par:p}
We now want to describe $\Lambda_I=\Lambda_{w_I}$.
Consider $(I\subset\overline{I}^c)\in\Schub(\Fl_{2n}(r,2n-r))$.
We draw $\Lambda_{(I,\overline{I}^c)}$ on Figure~\ref{fig:GammaSp} including the dotted part.

\begin{prop}
\label{prop:lambdaISp}
  \begin{enumerate}
  \item The part $\Lambda_{(I,\overline{I}^c)}$ is symmetric relatively to the dashed line.
  \item The part $\Lambda_I$ is the intersection of $\Lambda(\Gr_\omega(r,2n))$ and
$\Lambda_{(I,\overline{I}^c)}$.
  \end{enumerate}
\end{prop}

\begin{proof}
The first assertion is a direct consequence of  Lemma~\ref{lem:sym}. 
Consider $W$ as a subgroup of $S_{2n}$ as in Paragraph~\ref{par:rsSp}.
Then, $w_I$ is the element of  $S_{2n}$ corresponding to the Schubert class
${(I,\overline{I}^c)}$ in $\Schub(\Fl_{2n}(r,2n-r))$ as in Paragraph~\ref{par:Lambda}.
The second assertion follows. 
\end{proof}

\subsection{The case $\SO_{2n+1}$}
\label{sec:SO}

\paragraphe{\bf Root system.}
\label{par:rsSO}

Let $V$ be a $2n+1$-dimensional vector space 
and $\base=(e_1,\cdots,e_{2n+1})$ be a basis of $V$.
We denote by $(x_1,\cdots,x_{2n+1})$ the dual basis. 
If $i\in [1;2n+1]$, we set $\overline{i}=2n+2-i$.
Let $G$ be the special orthogonal group associated to the quadratic form 
$$
Q=x_{n+1}^2+\sum_{i=1}^nx_ix_{\overline{i}}.
$$ 

Set $T=\{{\rm diag}(t_1,\cdots,t_n,1,t_n^{-1},\cdots,t_1^{-1})\,:\,t_i\in\CC^*\}$.
Let $B$ the Borel subgroup of $G$ consisting of upper triangular matrices of $G$.
Let $\varepsilon_i$ denote the character of $T$ which maps 
$\diag(t_1,\cdots,t_n,1,t_n^{-1},\cdots,t_1^{-1})$ to $t_i$; we have 
$X(T)=\oplus_{i=1}^n\ZZ\varepsilon_i$.
Here, we have:
$$
\begin{array}{l}
  \Phi^+=\{\varepsilon_i\pm\varepsilon_j\,:\,1\leq i<j\leq n\}\cup 
\{\varepsilon_i\,:\,1\leq i\leq n\},\\
\Delta=\{\alpha_1=\varepsilon_1-\varepsilon_2,\,\alpha_2=\varepsilon_2-\varepsilon_3,\cdots,\,
\alpha_{n-1}=\varepsilon_{n-1}-\varepsilon_n,\,\alpha_n=\varepsilon_n\}.
\end{array}
$$

The Weyl group $W$ of $G$ is a subgroup of the Weyl group $S_{2n+1}$ of $\SL(V)$:
$$
W=\{w\in S_{2n+1}\,:\,w(\overline{i})=\overline{w(i)} \ \ \forall i\in [1;2n+1]\}.
$$
We will denote by $F(r)$ (resp. $\overline{F}(r)$) the span of $e_1,\cdots,e_r$ 
(resp. $e_{\overline{1}},\cdots,e_{\overline{r}}$).
We will denote by $V(r)$ the span of $e_{r+1},\cdots,e_{\overline{r+1}}$.
\\

\paragraphe{\bf Tangent space of orthogonal Grassmanians.}
Let $\alpha_r$ be a simple root, $P_r$ be the corresponding maximal parabolic 
subgroup of $G$ and $L_r$ be its Levi subgroup containing $T$.
For $r\leq n$, we denote by $\Gr_Q(r,2n+1)$ the orthogonal Grassmannian  of 
$r$-dimensional subspaces $M$ of $V$ such that $Q_{|M}=0$.
The homogeneous space $G/P_r$ with base point $P_r$ is 
$\Gr_Q(r,2n+1)$ with base point $F(r)$.

Note that $V=F(r)\oplus V(r)\oplus\overline{F}(r)$.
Moreover, $F(r)^{\perp_Q}=F(r)\oplus V(r)$, and $Q$ identifies $\overline{F}(r)$ 
with the dual of $F(r)$.
The tangent space $T_{F(r)}\Gr_Q(r,2n+1)$ identifies with 
${\rm Hom}(F(r),V(r))\oplus \bigwedge^2F(r)^*$.
The natural action of $L_r$ which is isomorphic to $S(\GL(F(r))\times{\rm O}(V(r)))$ 
makes this identification equivariant.

We set for $i\in [1,n]$, 
$\varepsilon_{\overline{i}}:=-\varepsilon_i$, and
$\varepsilon_{n+1}=0$.
Then, we have:
$$
\begin{array}{r@{\:}c@{\:}l@{\,}l@{\,}l}
\Phi^-&=&\{\varepsilon_i-\varepsilon_j\,:\, j&                      <i<\overline{j}\},\, &{\rm and}\\ 
\Lambda(\Gr_Q(r,2n+1))&=&\{\varepsilon_i-\varepsilon_j\,:\, j\leq r& <i<\overline{j}\}.
\end{array}
$$
We now represent each element of $\Lambda(\Gr_Q(r,2n+1))$ by a box on Figure~\ref{fig:GammaSO}.

\begin{figure}\psset{unit=5mm}
\centering
\begin{pspicture}(0,0)(4,7)
\multido{\n=4+1}{4}{
\psline(0,\n)(4,\n)
}
\multido{\n=1+1,\i=1+1}{3}{
\psline(0,\n)(\i,\n)
}
\psline(0,1)(0,7)
 
\multido{\i=1+1,\n=1+1}{4}{
\psline(\i,\n)(\i,7)
} 

\uput[u](0.5,7){$\varepsilon_1$}
\uput[u](2,7){$\dots$}
\uput[u](3.5,7){$\varepsilon_r$}
\uput[l](0,6.5){$\varepsilon_{r+1}$}
\uput[l](0,3.5){$\vdots$}
\uput[l](0,0.5){$\varepsilon_{2n+1}$}

\psline[linestyle=dashed](0,0)(7,7)
\psline[linestyle=dotted](0,0)(7,0)(7,4)(4,4)

\psline[linewidth=2pt](0,4)(4,4)
\end{pspicture}
\caption{Roots of $T_{F(r)}\Gr_Q(r,2n+1)$}
\label{fig:GammaSO}
\end{figure}

The boxes corresponding to roots of $\bigwedge^2F(r)^*$ (resp. $\Hom(F(r),V(r))$) are in the triangular 
(resp. rectangular) part of Figure~\ref{fig:GammaSO}.\\
 
\paragraphe{\bf Schubert varieties of orthogonal  Grassmanians.}
If $I\in\Part(r,2n+1)$ then we set $\overline{I}=\{\overline{i}\,:\,i\in I\}$ and
$$
\Schub(\Gr_Q(r,2n+1)):=\{I\in\Part(r,2n+1)\,:\, I\cap\overline{I}=\emptyset\}.
$$
The subspace  $F(I)$ belongs to $\Gr_Q(r,2n+1)$ if and only if 
$I\in \Schub(\Gr_Q(r,2n+1))$; so, the Schubert varieties $\Psi(I)$ of $\Gr_Q(r,2n+1)$ are indexed  
by $I\in\Schub(\Gr_Q(r,2n+1))$.
If $I=\{i_1<\cdots<i_r\}\in \Schub(r,2n+1)$, we set $i_{\overline{k}}=\overline{i_k}$ and write 
$(I\cup\overline{I})^c=\{i_{r+1}<\cdots<i_{\overline{r+1}}\}$. 
Then, the element of $W^{P_r}$ which corresponds to $\Psi(I)$ is
$w_I=(i_1,\cdots,i_{2n+1})$.\\

\paragraphe\label{par:pSO}
We now want to describe $\Lambda_I=\Lambda_{w_I}$.
Consider $(I\subset\overline{I}^c)\in
\Schub(\Fl_{2n+1}(r,2n+1-r))$.
We draw $\Lambda_{(I,\overline{I}^c)}$ on Figure~\ref{fig:GammaSO} including the dotted part.
Then, we obtain easily:

\begin{prop}
\label{prop:lambdaSO}
  \begin{enumerate}
  \item The part $\Lambda_{(I,\overline{I}^c)}$ is symmetric relatively to the dashed line.
  \item The part $\Lambda_I$ is the intersection of $\Lambda(\Gr_Q(r,2n+1))$ and
$\Lambda_{(I,\overline{I}^c)}$.
  \end{enumerate}
\end{prop}

\section{Levi-movability}
\label{sec:levimov}

In this section, we recall the Belkale-Kumar notion of Levi-movability (see~\cite{BK}).
We follow notation of Section~\ref{sec:paramschubert}.

\subsection{Cohomology of $G/P$}

\paragraphe
Let $\sigma_w$ denote the Poincar\'e dual of the homology class of $\Omega(w)$.
We have:
$$
{\rm H^*}(G/P,\ZZ)=\oplus_{w\in W^P} \ZZ\sigma_w.
$$
The dual of the class $\sigma_w$ is denoted by $\sigma_w^\vee$.
Note that $\sigma_e$ is the class of the point.
Let $\sigma_u,\,\sigma_v,\,\sigma_w$ be three Schubert classes (with
$u,v,w\in W^P$). 
If there exists an integer $d$ such that $\sigma_u.\sigma_v.\sigma_w=d\sigma_e$, 
we set $c_{uvw}=d$ and we set $c_{uvw}=0$ otherwise. 
These coefficients are the (symmetrized) structure coefficients 
of the cup product on 
${\rm H}^*(G/P,\ZZ)$ in the Schubert basis in the following sense:
$$
\sigma_u.\sigma_v=\sum_{w\in W^P} c_{uvw}\sigma_w^\vee;
$$
and $c_{uvw}=c_{vuw}=c_{uwv}$.\\

\paragraphe
Let $u,\,v$ and $w$ in $W^P$.
Let us consider the tangent space $T_u$ of 
the $u{-1}BuP/P$'s at the point $P$; and, similarly $T_v$ and $T_w$. 
Using the transversality theorem 
of Kleiman, Belkale-Kumar showed in \cite[Proposition~2]{BK} the 
following important lemma:

\begin{lemma}\label{lem:fondBK}
  The coefficient $c_{uvw}$ is nonzero if and only if there exist $p_u,p_v,p_w\in P$ such that 
the natural map
$$
T_P(G/P)\longto \frac{T_P(G/P)}{p_uT_u}\oplus \frac{T_P(G/P)}{p_vT_v}\oplus \frac{T_P(G/P)}{p_wT_w},
$$
is an isomorphism.
\end{lemma}

Then, Belkale-Kumar defined Levi-movability:\\

\begin{defin}
The triple $(\sigma_u,\,\sigma_v,\,\sigma_w)$ is said to be {\it Levi-movable} if there exist
 $l_u,l_v,l_w\in L$ such that the natural map
$$
T_P(G/P)\longto \frac{T_P(G/P)}{l_uT_u}\oplus \frac{T_P(G/P)}{l_vT_v}\oplus \frac{T_P(G/P)}{l_wT_w},
$$
is an isomorphism.
\end{defin}

We set:
$$
c_{uvw}^\kbprod=\left\{
  \begin{array}{ll}
c_{uvw}&{\rm \ if\ }  (\sigma_u,\,\sigma_v,\,\sigma_w) {\rm\ is\ Levi-movable;}\\
    0&{\rm\ otherwise.}
  \end{array}
\right .
$$

Note that in \cite{RR}, an equivalent characterization of Levi-movability is given.
We define on the group ${\rm H}^*(G/P,\ZZ)$ a bilinear product $\kbprod$ by the formula:
$$
\sigma_u\kbprod\sigma_v=\sum_{w\in W^P} c_{uvw}^\kbprod\sigma_w^\vee.
$$
By \cite[Definition~18]{BK}, we have:

\begin{theo}
  The product $\kbprod$ is commutative, associative and satisfies Poincar\'e duality.
\end{theo}

Remark that if $G/P$ is cominuscule, $P$ and $L$-orbits in $T_PG/P$
are equal. In particular, in this case the product $\kbprod$ is the
usual cup product.

\section{Cones associated to groups}
\label{section:Eigencone}

\subsection{The tensor product cone}

In this section, we will define a generalization of the Horn cone for any semisimple group $G$.
We will also recall some results about these cones.
We follow notation of Section~\ref{sec:levimov}.

\bigskip\paragraphe{\bf  The Borel-Weil theorem.} 
\label{par:Lnu}
Let $\nu$ be a character of $B$. Let $\CC_\nu$ denote the field $\CC$ 
endowed with the action of $B$ defined by
$b.\tau=\nu(b)\tau$ for all $\tau\in\CC_\nu$ and $b\in B$. The fiber product 
$G\times_B \CC_{-\nu}$ is a $G$-linearized 
line bundle on $G/B$, denoted  by $\Li_\nu$. 
In fact, the map $ X(B)= X(T)\longto\Pic^G(G/B),\,\nu\longmapsto\Li_\nu$ is an isomorphism.
Moreover, $\Li_\nu$ is generated by its sections if and only if it has nonzero sections
if and only if $\nu$ is dominant; and,
${\rm H}^0(G/B,\Li_\nu)$ is isomorphic to the dual $V_\nu^*$ of the irreducible $G$-module
$V_\nu$ of highest weight $\nu$. 

\bigskip\paragraphe
We set: $X(T)_\QQ=X(T)\otimes \QQ$.
The set of triples $(\lambda,\,\mu,\,\nu)\in (X(T)^+)^3$ such that 
$V_{\lambda}\otimes V_{\mu}\otimes V_{\nu}$ contains  nonzero $G$-invariant
vectors is a finitely generated semigroup.
We will denote by $\lr(G)$ the convex hull in $X(T)_\QQ^3$ of this semigroup: 
it is a closed convex rational polyhedral cone.\\

Set $X=(G/B)^3$. 
Identifying $X(T^3)$ with $X(T)^3$, for any $(\lambda,\,\mu,\,\nu)\in X(T)^3$, we obtain
a $G^3$-linearized line bundle $\Li_{\lambda,\mu,\nu}$ on $X$.
Applying the Borel-Weil theorem, we obtain
$$
\lr(G)=\{(\lambda~,\mu,\nu)\in X(T)^3\otimes \QQ\,:\,\exists n>0\ \ 
{\rm H}^0(X,\Li_{\lambda,\mu,\nu}^{\otimes n})^G\neq\{0\}\}.
$$
Since $G$ is assumed to be semisimple, we have isomorphisms
$X(T^3)_\QQ\simeq\Pic^{G^3}(X)_\QQ\simeq\Pic^G(X)_\QQ$.
With these identifications, 
$\lr(G)$ is the closure of $\ac^G((G/B)^3)$ (see for example \cite[Proposition~10]{GITEigen}).

\bigskip\paragraphe
Let $\alpha$ be a simple root of $G$,
$P_\alpha$ denote the associated maximal standard parabolic subgroup  
and $L_\alpha$  denote its Levi subgroup containing $T$.
Set $W_\alpha=W_{L_\alpha}$.
Consider the one parameter subgroup $\omega_{\alpha^\vee}$ (with usual notation) of the center of  $L_\alpha$.
We now state the main result of \cite{BK}:

\begin{theo}
\label{th:BK}
Here $G$ is assumed to be semisimple.
 Let $(\lambda,\mu,\nu)\in X(T)^3_\QQ$ dominant.
Then,  $(\lambda,\mu,\nu)\in \lr(G)$ if and only if 
\begin{eqnarray}
  \label{eq:ineq}
\langle u\omega_{\alpha^\vee},\lambda\rangle+
\langle v\omega_{\alpha^\vee},\mu\rangle+\langle w\omega_{\alpha^\vee},\nu\rangle\leq 0,
\end{eqnarray}
for all simple root $\alpha$ and all triple $(u,v,w)\in W/W_\alpha$ with $c^\kbprod_{uvw}=1$.
\end{theo}

Let $\alpha$ and $(u,v,w)\in W/W_\alpha$ be as in the theorem. 
The set of  $(\lambda,\mu,\nu)\in \lr(G)$ for which inequality~\eqref{eq:ineq} becomes an
equality is a face  of $\lr(G)$ denoted by $\Face(\alpha,u,v,w)$.
The following statement, proved in \cite{GITEigen}, shows that no inequality~\eqref{eq:ineq} can be omitted 
in Theorem~\ref{th:BK}.

\begin{theo}
\label{th:eigen1}
Let $\alpha$ and $(u,v,w)\in W/W_\alpha$ be as in Theorem~\ref{th:BK}.
Then, $\Face(\alpha,u,v,w)$ is a codimension one face of $\lr(G)$ intersecting
 the strictly dominant chamber.
\end{theo}

\bigskip\paragraphe\label{par:defrho}
We now want to understand better the faces $\Face(\alpha,u,v,w)$.
Consider the fixed point set $X^{\omega_{\alpha^\vee}}$ of $\omega_{\alpha^\vee}$ acting on $X$.
Then, 
$$C(u,v,w)=
L_\alpha u^{-1}B\times 
L_\alpha v^{-1}B\times L_\alpha w^{-1}B
$$
is an irreducible component of $X^{\omega_{\alpha^\vee}}$.
Note that $B_L=B\cap L_\alpha$ is a Borel subgroup of $L_\alpha$.
If each $w_i$ belongs to $W^P$, we fix an isomorphism
between  $(L_\alpha/B_L)^3$ and $C(u,v,w)$ by 
$$
(l_uB_L,l_vB_L,l_wB_L)\longmapsto
l_uu^{-1}B\times l_vv^{-1}B\times l_ww^{-1}B;
$$
it is well-defined since $u,\,v,\,w\in W^P$.
In particular, the group $\Pic^{L_\alpha^3}(C(u,v,w))$ is isomorphic to $\Pic^{L_\alpha^3}((L_\alpha/B_L)^3)$; 
that is, to $X(T)^3$. With these identifications the restriction morphism 
$\Pic^{G^3}(X)\longto\Pic^{L_\alpha^3}(C(u,v,w))$ is
$$
\begin{array}{cccc}

\rho_{uvw}\,:&X(T)^3&\longto& X(T)^3\\
                   &(\lambda,\mu,\nu)&\longmapsto&(u^{-1}\lambda,v^{-1}\mu,w^{-1}\nu).
\end{array}
$$

The following statement is \cite[Lemma~1]{carquois}:

\begin{theo}
\label{th:facedom}
Let $\alpha$ and $(u,v,w)\in W/W_\alpha$ be such that
$\sigma_{u}.\sigma_{v}.\sigma_{w}\neq 0$.
Then, for any  $(\lambda,\mu,\nu)\in \lr(G)$, 
$$
\langle u\omega_{\alpha^\vee},\lambda\rangle+
\langle v\omega_{\alpha^\vee},\mu\rangle+\langle w\omega_{\alpha^\vee},\nu\rangle\leq 0,
$$
holds.
Let $\Face(\alpha,u,v,w)$ denote the corresponding face of $\lr(G)$.
If $(\lambda,\mu,\nu)\in X(T)^3\otimes\QQ$ is dominant
then  $(\lambda,\mu,\nu)\in \Face(\alpha,u,v,w)$ if and only if
$ \rho_{uvw}(\lambda,\mu,\nu)\in\lr(L_\alpha)$.
\end{theo}

The following criterion to decide if $c_{uvw}^\kbprod=1$ or not
will play a central role in the sequence of this article.

\begin{coro}
\label{cor:face1}
  Let $\alpha$ and $(u,v,w)\in W/W_\alpha$ be as in Theorem~\ref{th:facedom}.
Then, $c_{uvw}^\kbprod=1$ if and only if
$\Face(\alpha,u,v,w)$ intersects the interior of the dominant chamber of $X(T^3)_\QQ$.
\end{coro}

\begin{proof}
The direct implication is a consequence of Theorem~\ref{th:eigen1}.
Conversely,
the cone $\lr(L_\alpha)$ 
has codimension one (the rank of 
the center of $L_\alpha$) in $X(T)^3_\QQ$.
So, since $\Face(\alpha,u,v,w)$ intersects the interior of the dominant chamber of $X(T^3)_\QQ$,
Theorem~\ref{th:facedom} implies that $\Face(\alpha,u,v,w)$ has codimension one.
So, the corresponding inequality has to appear in Theorem~\ref{th:BK}. 
This implies that $c_{uvw}^\kbprod=1$.
\end{proof}

\subsection{The eigencone}
\label{sec:eigencone}

Let us fix a maximal compact subgroup $U$ of $G$ in such a way that $T\cap U$ is a Cartan subgroup of $U$.
Let $\lu$ and $\lt$ denote the Lie algebras of $U$ and $T$.
Let $\lt^+$ be the Weyl chamber of $\lt$ corresponding to $B$.
Let $\sqrt{-1}$ denote the usual complex number.
It is well known that $\sqrt{-1}\lt^+$ is contained in $\lu$ and that the map:
$$
\begin{array}{ccc}
  \lt^+&\longto&\lu/U\\
\xi&\longmapsto&U.(\sqrt{-1}\xi)
\end{array}
$$
is an homeomorphism.
Consider the set
$$
\Gamma(U):=\{(\xi,\zeta,\eta)\in(\lh^+)^3\,:\,
U.(\sqrt{-1}\xi)+U.(\sqrt{-1}\zeta)+U.(\sqrt{-1}\eta)\ni 0\}.
$$

Let $\lu^*$ (resp. $\lt^*$) denote the dual (resp. complex dual) of $\lu$ (resp. $\lt$).
Let $\lt^{*+}$ denote the dominant chamber of $\lt^*$ corresponding to $B$.
By taking the tangent map at the identity, one can embed $X(T)^+$ in $\lt^{*+}$.
Note that, this embedding induces a rational structure on the complex vector space $\lt^*$.
In particular, we can embed $\lr(G)$ in $(\lt^{*+})^3$:
let $\tilde\lr(G)$ denote the so obtained part of $\lt^{*+}$.

Now, using the Cartan-Killing form, we identify $\lt^+$ and $\lt^{*+}$.
In particular, we can embed $\Gamma(U)$ in $(\lt^{*+})^3$; the so obtained cone is
denoted by $\tilde\Gamma(U)$.

\begin{theo}
\label{th:KN}
The set $\Gamma(U)$ is a closed convex  polyhedral cone. 
Moreover, $\tilde\lr(G)$ is the set of  the rational points of $\tilde\Gamma(U)$.
\end{theo}

\section{About the cohomology of $\Gr_\omega(r,2n)$}
\label{sec:cohomSp}

\bigskip
\paragraphe
This section is concerned by coefficient structures of the cohomology of ordinary and isotropic Grassmannians. 
To avoid any confusion, those concerning ordinary and isotropic Grassmanians will be denoted by $c$ and $d$ 
respectively. 
Note that, since ordinary Grassmannian is cominuscule, $c^\kbprod=c$.\\

In Paragraph \ref{par:p}, we defined combinatorially an injective map

\begin{eqnarray}
  \label{eq:SS1}
\begin{diagram}[height=17pt]
  \Schub(\Gr_\omega(r,2n))&\rInto&\Schub(\Fl_{2n}(r,2n-r)).\\
I&\longmapsto&(I,{\bar I}^c)\\
\end{diagram}
\end{eqnarray}

Set $\Schub(\Gr(r,n))=\Part(r,n)$.
 In Paragraph \ref{par:SFl}, we defined combinatorially an injective map 

 \begin{eqnarray}
   \label{eq:SS2}
\begin{diagram}[height=17pt]
  \Schub(\Fl_{2n}(r,2n-r))&\rInto&\Schub(\Gr(r,2n-r))\times
 \Schub(\Gr(2(n-r),2n-r))\times  \Schub(\Gr(r,2r)).\\
(I^1,I^2)&\longmapsto&(I_0,I_1,I_2)\\
\end{diagram}
 \end{eqnarray}

 By composing these two injective maps and then forgetting $I_1$, we obtain an injective  map 
\begin{eqnarray}
   \label{eq:SS3}
\begin{diagram}[height=17pt]
  \Schub(\Gr_\omega(r,2n))&\rInto&\Schub(\Gr(r,2n-r))\times \Schub(\Gr(r,2r)).
  \\
I&\longmapsto&(I_0,I_2)\\
\end{diagram}
\end{eqnarray}

The aim of this section is to prove that this immersion is relevant relatively to
the Belkale-Kumar product. We will also use the following particular
case of the construction in  Paragraph \ref{par:p}:
\begin{eqnarray}
   \label{eq:SS4}
\begin{diagram}[height=17pt]
  \Schub(\Gr_\omega(r,2r))&\rInto&\Schub(\Gr(r,2r)).\\
I&\longmapsto& I=I_2\\
\end{diagram}
\end{eqnarray}

\bigskip
\paragraphe
The following result is due to Belkale-Kumar:

\begin{prop}
\label{prop:nonzeroIG}
Let $I$, $J$, $K\in\Schub(\Gr_\omega(r,2n))$ such that 
$|\Lambda_I|+|\Lambda_J|+|\Lambda_K|=\dim\Gr_\omega(r,2n)$.
Notations $I_0, J_0, K_0, I_2, J_2$ and $K_2$ refer to the map
\eqref{eq:SS3}.
Note also that $I_2,\,J_2$ and $K_2$ belong to  $\Schub(\Gr_\omega(r,2r))$.

The following are equivalent:
\begin{enumerate}
\item $d^\kbprod_{IJK}\neq 0$; 
\item $|\Lambda_{I_0}|+|\Lambda_{J_0}|+|\Lambda_{K_0}|=2r(n-r)$ and
$d_{IJK}\neq 0$;
\item  
$d_{I_2J_2K_2}\neq 0$ and 
$c_{I_0\,J_0\,K_0}\neq 0$.
\end{enumerate}
\end{prop}

\begin{remark}
  The first assertion concerns a structure coefficient of
  $(H^*(\Gr_\omega(r,2n),\kbprod)$, the second one concerns
  $H^*(\Gr_\omega(r,2n)$ and the last one concerns structure coefficients of
  $H^*(\Gr_\omega(r,2r)$
and $H^*(\Gr(r,2n-r)$.
\end{remark}

\begin{proof}
  This is essentially \cite[Theorem~30]{BK:typeBC}. We include a brief discussion for completeness.

The equivalence between the two first assertions is \cite[Proposition~2.4]{RR}.
We use notation of Section~\ref{sec:expleSp} for $\Sp_{2n}$.
Consider the decomposition of $T_{P_r}\Gr_\omega(r,2n)$ as sum of irreducible $L_r$-modules.
The centered tangent space of $\Omega_I(\Gr_\omega(r,2n))$ decomposes as the sum of those 
of $\Omega_{I_0}(\Gr(r,2n-r))$ and those of 
 $\Omega_{I_2}(\Gr_\omega(r,2r))$.
Since $(I,J,K)$ is Levi-movable, one immediately deduces that
$(I_2,J_2,K_2)$ and $(I_0,J_0,K_0)$ are. 
In particular, Lemma~\ref{lem:fondBK} implies that $d_{I_2J_2K_2}\neq 0$ and 
$c_{I_0\,J_0\,K_0}\neq 0$.

The fact that the last assertion implies the second one is the difficult part of 
\cite[Theorem~30]{BK:typeBC}.
\end{proof}

\bigskip\paragraphe
Here comes our main result about cohomology of $\Gr_\omega(r,2n)$; it allows to characterize 
the condition $d^\kbprod_{IJK}=1$ in terms of the Littlewood-Richardson coefficients.

\begin{theo}
\label{th:ppalIG} 
 Let $I$, $J$, $K\in\Schub(\Gr_\omega(r,2n))$ such that 
$|\Lambda_{I}|+|\Lambda_{J}|+|\Lambda_{K}|=\dim\Gr_\omega(r,2n)$.
Notations $I_0, J_0, K_0, I_2, J_2$ and $K_2$ refer to the map
\eqref{eq:SS3}.
The following are equivalent:
\begin{enumerate}
\item $d^\kbprod_{IJK}=1$;\label{ass:IGBK}
\item \label{ass:IG}
$d_{I_2J_2K_2}=1$ and
$c_{I_0\,J_0\,K_0}=1$;
\item  \label{ass:IGLR}
$c_{I_2J_2K_2}=1$ and
$c_{I_0\,J_0\,K_0}=1$.
\end{enumerate}
\end{theo}

\begin{proof}
We first prove that assertion~\ref{ass:IGBK} implies assertion~\ref{ass:IG}.
Proposition~\ref{prop:nonzeroIG} implies that $d_{I_2J_2K_2}\neq 0$ and 
$c_{I_0\,J_0\,K_0}\neq 0$.
Now, by Corollary~\ref{cor:face1}, it is sufficient to prove that the two faces $\Face_2$
and $\Face_0$ of  $\lr(\Sp_{2r})$ and $\lr(\GL_{2n-r})$ corresponding 
to these coefficients intersect the strictly dominant chambers.
We are going to prove this by constructing explicitly matrices whose
the corresponding spectrum yield points of $\Face_2$
and $\Face_0$. The starting point is that the assumption
$d^\kbprod_{IJK}=1$ yields matrices whose spectrum have certain properties.

We first make more explicit the description of the face $\Face(r,I,J,K)$ 
of $\lr(\Sp_{2n})$ associated to  $d^\kbprod_{IJK}=1$  as in
Theorem~\ref{th:eigen1}.
By Theorem \ref{th:eigen1}, there exists a point $(\lambda,\mu,\nu)$
  in  $\Face(r,I,J,K)$ such that $\lambda$, $\mu$ and $\nu$ are
  strictly dominant.
Let us use notation of Section~\ref{sec:expleSp} for the data associated to the group $\Sp_{2n}$.
We write $\lambda=\sum\lambda_i\varepsilon_i\in X(T)$, we
recall that $\overline{i}=2n+1-i$ and we set $\lambda_{\overline
  i}=-\lambda_i$ for $i\in [1,n]$.
We use similar notation for $\mu$ and $\nu$. 
A direct computation shows the linear equation of $\Face(r,I,J,K)$ is
\begin{eqnarray}
  \label{eq:phiI}
  \sum_{i\in I}\lambda_i+\sum_{j\in J}\mu_j+\sum_{k\in K}\nu_k=0.
\end{eqnarray}

Let us consider the isomorphism $\rho_{IJK}$ of $X(T)^3$ defined in
Paragraph \ref{par:defrho}.
By Theorem \ref{th:facedom},  $\rho_{IJK}(\Face(r,I,J,K))\subset \lr(L_r)$. 
Using Theorem \ref{th:KN}, we now identify $\lr(L_r)$ with
$\Gamma(L_r\cap U_{2n}(\CC))$.
We now want to describe a point in $\Gamma(L_r\cap U_{2n}(\CC))$
corresponding to $\rho_{IJK}(\lambda,\mu,\nu)$.

The elements of $Lie(L_r)$  have the following form:
\begin{eqnarray}
  \label{eq:matLr}
A=\left(
  \begin{array}{c|c|c}        
    A_1&0&0\\
\hline
0&A_2&0\\
\hline
0&0&-(J_r\,{}^tA_1\, J_r)
  \end{array}
\right),
\end{eqnarray}
where $A_1\in Lie(\GL_{r})$ and $A_2\in Lie(\Sp_{2(n-r)})$ and $J_r$ is defined by
formula \eqref{eq:defJn} (with $r$ in place of $n$).
Let $\lt^+_r$ be the dominant (relatively to  $B\cap L_r$) chamber of the Cartan subalgebra $\lt$
of $L_r$. 
Note that $\lt^+_r$ is the set of diagonal real matrices
$\diag(\alpha_1,\cdots,\alpha_n,-\alpha_n,\cdots,-\alpha_1)$ such that 
$\alpha_1\geq\cdots\geq \alpha_r$ and $\alpha_{r+1}\geq\cdots\geq \alpha_n$.
 Recall that we have an homeomorphism 
$\pi\,:\, \lt^+_r \longto Lie(L_r\cap U_{2n}(\CC))/ (L_r\cap U_{2n}(\CC))$.
Let $A\in Lie(L_r\cap U_{2n})$ as in formula \eqref{eq:matLr}.
By juxtaposition of the spectrums of $\sqrt{-1}A_1,\,\sqrt{-1}A_2$ and
$-\sqrt{-1} (J_r\,{}^tA_1\, J_r)$ 
(each one in non-increasing order), we obtain a point $\xi(A)$ in
$\lt^+_r\subset\RR^{2n}$. Note that $\pi(\xi(A))=(L_r\cap U_{2n}(\CC)).A$.

Fix $A,\,B$ and $C$ in $Lie(L_r\cap U_{2n})$ such that 
$$
  A+B+C=0,
$$
and 
$(\xi(A),\xi(B),\xi(C))$ is the point of $\lr(L_r)$ corresponding  to 
$\rho_{IJK}(\lambda,\mu,\nu)$.
The matrices $A$, $B$ and $C$ are as in formula \eqref{eq:matLr} for
some 
;
$A_1,B_1,C_1\in u_r(\CC)$ and $A_2,B_2,C_2\in
u_{2(n-r)}(\CC)\cap{\rm Lie}(\Sp_{2(n-r)})$.

\bigskip
Consider now the three following matrices of $\Sp_{2r}\cap U_{2r}$:
$$
\bar A=\left(
  \begin{array}{c|c}
    A_1&0\\
\hline
0&-(J_r\,{}^tA_1\, J_r)
  \end{array}
\right),\,
\bar B=\left(
  \begin{array}{c|c}
    B_1&0\\
\hline
0&-(J_r\,{}^tB_1\, J_r)
  \end{array}
\right),$$

and,
$$
\bar C=\left(
  \begin{array}{c|c}
    C_1&0\\
\hline
0&-(J_r\,{}^tC_1\, J_r)
  \end{array}
\right).
$$
Obviously, $\bar A+\bar B+\bar C=0$ and the spectrum of these matrices
yield a point of $\Gamma(\Sp_{2r}\cap U_{2r})$.
We claim that the corresponding point (by Theorem \ref{th:KN}) is
regular and belongs to $\Face_2$.

Let  $\alpha=(\alpha_1,\cdots,\alpha_n,-\alpha_n,\cdots,-\alpha_1)$ be the
spectrum of $\sqrt{-1}A$; it satisfies  
$\alpha_1>\cdots> \alpha_r$ and $\alpha_{r+1}>\cdots> \alpha_n$.
Recall that $w_{I}\in S_{2n}$.
Moreover, $w_I\alpha$ corresponds to $\lambda$ and so is strictly
dominant.
Consider now, $w_{I_2}\in S_{2r}$.
Since $\lambda$ is dominant, so is its restriction $\bar\lambda$ to $T\cap Sp(2r)$.
So, the coordinates of 
$w_{I_2}(\alpha_1,\cdots,\alpha_r,-\alpha_r,\cdots,\-\alpha_1)$ form a
decreasing sequence.
This implies that
$$
\sum_{i\in I_2}(\bar\lambda)_i=\sum_{i=1}^r\alpha_i=tr(A_1).
$$
Similarly, we have
$$
\sum_{i\in J_2}(\bar\mu)_i=tr(B_1) {\rm\ and\ }
\sum_{i\in K_2}(\bar\nu)_i=tr(C_1),
$$
with obvious notation.
Now, the relation $\bar A+\bar B+\bar C=0$ implies that 
$$
\sum_{i\in I_2}(\bar\lambda)_i+\sum_{i\in J_2}(\bar\mu)_i+\sum_{i\in K_2}(\bar\nu)_i=0.
$$
So, $(\bar\lambda,\bar\mu,\bar\nu)$ is a regular point in $\Face_2$
and Corollary \ref{cor:face1} implies that $d_{I_2J_2K_2}=1$.

In a similar way,
$$
\left(
  \begin{array}{c|c}
    A_1&0\\
\hline
0&A_2
  \end{array}
\right)+\left(
  \begin{array}{c|c}
    B_1&0\\
\hline
0&B_2
  \end{array}
\right)+\left(
  \begin{array}{c|c}
    C_1&0\\
\hline
0&C_2
  \end{array}
\right)=0,
$$
provides a regular point in $\Face_0$.
So, Corollary \ref{cor:face1} implies that $c_{I_0J_0K_0}=1$.\\

We now prove that  assertion~\ref{ass:IG} implies assertion~\ref{ass:IGLR}.
This implication is only concerned about $\Gr(r,2r)$ and $\Gr_\omega(r,2r)$: we may assume that $r=n$.
 Let us assume that $d_{IJK}=d_{IJK}^\kbprod=1$.
By \cite[Corollary~11]{BK:typeBC}, the following product in ${\rm H}^*(\Gr(n,2n))$ is nonzero:
$$
\sigma_I(\Gr(n,2n)).\sigma_{J}(\Gr(n,2n)).\sigma_{K}(\Gr(n,2n))\neq 0.
$$
Now, by Corollary~\ref{cor:face1} it is sufficient to prove that the face
$\Face^A$ of $\lr(\SL_{2n})$ corresponding to $(I,J,K)$ contains regular points.
Let $\Face^C$ be the  face of $\lr(\Sp_{2n})$ corresponding to $d_{IJK}=1$.
By Theorems~\ref{th:eigen1} and \ref{th:KN}, there exist $A,B,C\in u_n(\CC)$ such that
$$
\left(
  \begin{array}{c|c}
    A&0\\
\hline
0&-J_n{}^tAJ_n
  \end{array}
\right)+\left(
  \begin{array}{c|c}
    B&0\\
\hline
0&-J_n{}^tBJ_n
  \end{array}
\right)+\left(
  \begin{array}{c|c}
    C&0\\
\hline
0&-J_n{}^tCJ_n
  \end{array}
\right)=0,
$$
and the spectrum $(\alpha,\beta,\gamma)$ of these three matrices give a regular point 
in $\Face^C$. Since
$$
\tr(A)+\tr(B)+\tr(C)=\sum_{I}\alpha_i+\sum_{J}\beta_i+\sum_{K}\gamma_i=0,
$$
we just obtained a regular point in $\Face^A$.\\
 
Still assuming that $r=n$, we now want to prove that
assertion~\ref{ass:IGLR} implies assertion~\ref{ass:IG}.
Consider the inclusion of $\Gr_\omega(n,2n)$ in $\Gr(n,2n)$.
Let $\Omega_I(\Gr(n,2n))$, $\Omega_{J}(\Gr(n,2n))$ and $\Omega_{K}(\Gr(n,2n))$ be the three Schubert
varieties of $\Gr(n,2n)$ corresponding to $I$, $J$ and $K$ and the standard flag in the basis of 
Paragraph~\ref{par:rsSp}.
Since $c_{IJK}=1$, \cite[Theorem~2]{Sottile:transv} implies that for general $g$, $g'$ and $g''$ in 
$\Sp_{2n}$  the intersection 
$g\Omega_I(\Gr(n,2n))\cap g'\Omega_{J}(\Gr(n,2n))\cap g''\Omega_{K}(\Gr(n,2n))$ is transverse and reduced 
to  one point $F$.
Let us consider the orthogonal $F^{\perp_\omega}$ of $F$ for $\omega$.
Since $g\in \Sp_{2n}$,  $F^{\perp_\omega}$ belongs to $g\Omega_I(\Gr(n,2n))$; and finally to
the intersection. We deduce that $F=F^{\perp_\omega}$ belongs to $\Gr_\omega(n,2n)$.
So, the intersection 
$g\Omega_I(\Gr_\omega(n,2n))\cap g'\Omega_{J}(\Gr_\omega(n,2n))\cap g''\Omega_{K}(\Gr_\omega(n,2n))$ 
is reduced  to one point $F$ for general  $g$, $g'$ and $g''$ in 
$\Sp_{2n}$. 
We deduce that $d_{IJK}=1$.\\

It remains to prove that assertion~\ref{ass:IGLR} implies assertion~\ref{ass:IGBK}.
By the preceding argue, assertion~\ref{ass:IG} holds. 
Since $\Gr_\omega(r,2r)$ is cominuscule, we may assume that $r<n$.
Now, Proposition~\ref{prop:nonzeroIG} implies that $d_{IJK}\neq 0$. 
It remains to prove that the corresponding face $\Face(r,I,J,K)$ of $\lr(\Sp_{2n})$ contains regular
points.
Let us consider the three Schubert classes $\sigma_{(I,{\bar I}^c)} (\Fl_{2n}(r,2n-r))$, $\sigma_{(J,{\bar J}^c)}(\Fl_{2n}(r,2n-r))$ and 
$\sigma_{(K,{\bar K}^c)}(\Fl_{2n}(r,2n-r))$ of $H^*(\Fl_{2n}(r,2n-r))$.
Since $c_{I_2J_2K_2}\neq 0$ and $c_{I_0\,J_0\,K_0}\neq 0$, the triple
$((I,{\bar I}^c),(J,{\bar J}^c),(K,{\bar K}^c))\in \Schub(\Fl_{2n}(r,2n-r)) $ is Levi-movable.
Let $d$ be the  positive integer such that
$$\sigma_{(I,{\bar I}^c) }(\Fl_{2n}(r,2n-r))\kbprod\sigma_{(J,{\bar J}^c)}(\Fl_{2n}(r,2n-r))
\kbprod\sigma_{(K,{\bar K}^c)}(\Fl_{2n}(r,2n-r))=d[{\rm pt}].$$

By \cite{Richmond:recursion} (see also \cite{Rich:mult} or \cite{multi}), $d$ is the product of 
$c_{I_0\,J_0\,K_0}$ and another Littlewood-Richardson coefficient $c$.
The fact that $c_{I_0\,J_0\,K_0}=1$ allows to apply Theorem~\ref{th:prod} to $c$:
$c=c_{I_2J_2K_2}.c_{I_0\,J_0\,K_0}$. We deduce that $d=1$.

By \cite{GITEigen2}, by saturating the two inequalities $\varphi_{IJK}$ and 
$\varphi_{\overline{I}^c\,\overline{J}^c\,\overline{K}^c}$, one obtains a face $\Face$ of $\lr(\SL_{2n})$ 
intersecting the strictly dominant chamber and of codimension two.

Let $T^A$ be the diagonal maximal torus of $\SL_{2n}$.
Let $\theta$ be the $\ZZ$-linear involution of $X(T^A)$ mapping $\varepsilon_i$ on $-\varepsilon_{2n+1-i}$,
with notation of Paragraph~\ref{par:defrootSL}.
Since $\theta$ corresponds to duality for representations, $\lr(\SL_{2n})$ is stable by 
the automorphism $(\theta,\theta,\theta)$ of $X(T^A)^3\otimes \QQ$.
Note that the character group of the maximal torus of $\Sp_{2n}$ defined in Paragraph~\ref{par:rsSp}
identifies by restriction with the set of $\theta$-fixed points in $X(T^A)$.
Moreover,  by \cite[Theorem~1]{BK:typeBC},  $\lr(\Sp_{2n})$ is precisely the set of  points in $\lr(\Sl_{2n})$
fixed by $(\theta,\theta,\theta)$.

Since $\varphi_{IJK}\circ(\theta,\theta,\theta)=\varphi_{\overline{I}^c\,\overline{J}^c\,\overline{K}^c}$, $\Face$ is stable 
by $(\theta,\theta,\theta)$. By convexity $\Face$ contains regular $\theta$-fixed points.
We deduce using \cite[Theorem~1]{BK:typeBC}, that $\Face(r,I,J,K)$ contains regular points.
\end{proof}

\subsection{Examples}

We now give some examples performed with the Anders~Buch's quantum calculator \cite{qcalc}.\\

\paragraphe
Several multiplicative formulas for structure constants of $\kbprod$ are known 
(see~\cite{Rich:mult,Richmond:recursion,multi,KnPu:product}).
The formula $d^\kbprod_{IJK}=d_{I_2J_2K_2}.c_{I_0\,J_0\,K_0}$ could explain
 Theorem~\ref{th:ppalIG}. Unfortunately, this last formula is not satisfied:

if $r=3$, $n=5$ and $I=J=K=\{3,\,7,\,10\}$ then $d^\kbprod_{IJK}=2$,
$d_{I_2J_2K_2}=2$ and
$c_{I_0\,J_0\,K_0}=2$.\\

\paragraphe
We now consider  $\IG(n,2n)$ and observe relations between 
$d_{IJK}$  and $c_{IJK}$ for $I,J,K\in\Sch(\IG(n,2n))\subset \Part(n,2n)=\Sch(\Gr(n,2n))$.
Since $\IG(n,2n)$ and $\Gr(n,2n)$ are cominuscule, the Belkale-Kumar product and the ordinary
one coincide here.
Let $\delta_I$ denote the number of diagonal elements in $\Lambda_I(\IG(n,2n))$.
Theorem~\ref{th:ppalIG} shows that
$$
d_{IJK}=1 \iff c_{IJK}=1.
$$
Assume that $d_{IJK}=1$. 
The fact that $c_{IJK}$ is nonzero implies that the sum of the codimensions of the three
corresponding Schubert varieties of $\Gr(n,2n)$ equals the dimension of $\Gr(n,2n)$.
One can easily check that this means that $\delta_I+\delta_J+\delta_K=n$.
The following example shows that this is not true if $d_{IJK}$ is only assumed to be nonzero:

Set $n=4$,  $I=\{1,\,2,\,4,\,6\}$ and $J=K=\{4,\,6,\,7,\,8\}$. 
Then $d_{IJK}=2$ and $\delta_I+\delta_J+\delta_K=3+1+1=5$.
In particular, $c_{IJK}=0$.

\bigskip\paragraphe
For $I,J,K$ in $\Sch(\IG(n,2n))$ such that $c_{IJK}=1$, we obviously have 
$\delta_I+\delta_J+\delta_K=n$. The following example shows that this is not true if
$c_{IJK}$ is only assumed to be nonzero.

Set $n=4$,  $I=J=\{2,\,4,\,6,\,8\}$ and $K=\{3,\,4,\,7,\,8\}$.  
Then $c_{IJK}=2$ and $\delta_I+\delta_J+\delta_K=6$. 
In particular, $d_{IJK}=0.$

\bigskip\paragraphe
We now assume that $\delta_I+\delta_J+\delta_K=n$ and 
$|\Lambda_I(\IG(n,2n))|+|\Lambda_J(\IG(n,2n))|+|\Lambda_K(\IG(n,2n))|=\frac{n(n+1)}{2}$.
The Belkale-Kumar-Sottile theorem (see~\cite[Theorem~2]{Sottile:transv}) implies that 
$$
\begin{array}{l}
  c_{IJK}\geq d_{IJK}{\rm\ and}\\
  c_{IJK}-d_{IJK} {\rm\ is\ even}.
\end{array}
$$
We already noticed that $c_{IJK}$ and $d_{IJK}$ can be different for dimension reasons. 
The following example shows that they can be different for other reasons.

Set $n=5$, $I=J=\{2,\,4,\,6,\,8,\,10\}$ and $K=\{3,\,6,\,7,\,9,\,10\}$.  Then
$$
d_{IJK}=4 {\rm\ \ and\ \ }c_{IJK}=6.
$$ 

\section{About the cohomology of $\Gr_Q(r,2n+1)$}
\label{sec:cohomSO}

This section is concerned by coefficient structures of the cohomology of ordinary 
and orthogonal Grassmanians. 
To avoid any confusion, those concerning ordinary and isotropic Grassmanians 
will be denoted with $c$ and $e$ respectively. 

With notation of Section \ref{sec:SO}, consider the injective maps

\begin{eqnarray}
  \label{eq:SS6}
\begin{diagram}[height=17pt]
  \Schub(\Gr_Q(r,2n+1))&\rInto&\Schub(\Fl_{2n+1}(r,2n+1-r)).\\
I&\longmapsto&(I,{\bar I}^c),\\
\end{diagram}
\end{eqnarray}

and

 \begin{eqnarray}
   \label{eq:SS7}
\begin{diagram}[height=17pt]
  \Schub(\Fl_{2n+1}(r,2n+1-r))&\rInto&\Schub(\Gr(r,2n+1-r))\times
 \Schub(\Gr(2(n-r)+1,2n+1-r))\times  \Schub(\Gr(r,2r)).\\
(I^1,I^2)&\longmapsto&(I_0,I_1,I_2)\\
\end{diagram}
 \end{eqnarray}

 By composing these two injective maps and then forgetting $I_1$, we obtain an injective  map 
\begin{eqnarray}
   \label{eq:SS8}
\begin{diagram}[height=17pt]
  \Schub(\Gr_Q(r,2n+1))&\rInto&\Schub(\Gr(r,2n+1-r))\times \Schub(\Gr(r,2r)).
  \\
I&\longmapsto&(I_0,I_2)\\
\end{diagram}
\end{eqnarray}

The aim of this section is to prove that this immersion is relevant relatively to
the Belkale-Kumar product. We will also use the following particular
case of the construction in  Paragraph \ref{par:p}:
\begin{eqnarray}
   \label{eq:SS9}
\begin{diagram}[height=17pt]
  \Schub(\Gr_\omega(r,2r))&\rInto&\Schub(\Gr(r,2r)).\\
I&\longmapsto& I=I_2\\
\end{diagram}
\end{eqnarray}

\bigskip\paragraphe
The following is \cite[Theorem~41]{BK:typeBC} :

\begin{prop}
\label{prop:nonzeroOG}
Let $I$, $J$, $K\in\Schub(\Gr_Q(r,2n+1))$ such that 
$|\Lambda_{I}|+|\Lambda_{J}|+|\Lambda_{K}|=\dim\Gr_Q(r,2n+1)$.
The following are equivalent:
\begin{enumerate}
\item $e^\kbprod_{IJK}\neq 0$;
\item $|\Lambda_{I_0}|+|\Lambda_{J_0}|+|\Lambda_{K_0}|=r(2n+1-2r)$ and
$e_{IJK}\neq 0$;
\item $e_{I_2J_2K_2}\neq 0$ and 
$c_{I_0\,J_0\,K_0}\neq 0$.
\end{enumerate}
\end{prop}

\bigskip\paragraphe
Here comes our main result about cohomology of $\Gr_Q(r,2n+1)$; it allows to characterize 
the condition $e^\kbprod_{IJK}=1$ in terms of the Littlewood-Richardson coefficients.

\begin{theo}
\label{th:ppalSO} 
 Let $I$, $J$, $K\in\Schub(\Gr_Q(r,2n+1))$ such that 
$|\Lambda_{I}|+|\Lambda_{J}|+|\Lambda_{K}|=\dim\Gr_Q(r,2n+1)$.
The following are equivalent:
\begin{enumerate}
\item $e^\kbprod_{IJK}=1$;\label{ass:OGBK}
\item \label{ass:OG}
$e_{I_2J_2K_2}=1$ and
$c_{I_0\,J_0\,K_0}=1$;
\item  \label{ass:OGLR}
$c_{I_2p'_2K_2}=1$ and
$c_{I_0\,J_0\,K_0}=1$.
\end{enumerate}
\end{theo}

\begin{proof}
  The proof which is similar to those of Theorem~\ref{th:ppalIG} is left to the reader. 
\end{proof}

 \bibliographystyle{amsalpha}
\bibliography{algo}

\begin{center}
  -\hspace{1em}$\diamondsuit$\hspace{1em}-
\end{center}

\end{document}